\documentclass{amsart}
\usepackage{lmodern}
\usepackage[T1]{fontenc}
\usepackage[english]{babel}
\usepackage[utf8]{inputenc}
\usepackage{amsmath}
\usepackage{graphicx}
\usepackage{amssymb}
\usepackage{csquotes}
\usepackage{amsthm}
\usepackage{tikz-cd}
\usepackage{mathrsfs}
\usepackage[colorinlistoftodos]{todonotes}
\usepackage{enumitem}
\usepackage{yfonts}
\usepackage{xcolor}
\usepackage{mathtools}
\usepackage{hyperref}
\usepackage{comment}
\usepackage{xspace}

\usepackage{microtype}

\numberwithin{equation}{section}

\newtheorem{thm}{Theorem}[section]
\newtheorem{lemma}[thm]{Lemma}
\newtheorem{prop}[thm]{Proposition}

\newtheorem{cor}[thm]{Corollary}

\newtheorem{rmk}[thm]{Remark}

\newtheorem*{rmk*}{Remark}
\newtheorem*{defn*}{Definition}

\newcommand{\bR}{\mathbf{R}}
\newcommand{\bG}{\mathbf{G}}

\newcommand{\R}{\mathbb{R}}
\newcommand{\Q}{\mathbb{Q}}
\newcommand{\C}{\mathbb{C}}
\newcommand{\Z}{\mathbb{Z}}
\newcommand{\N}{\mathbb{N}}
\newcommand{\F}{\mathbb{F}}

\newcommand{\T}{\mathbb{T}}

\newcommand{\Ch}[1]{\mathrm{Ch}\left(#1\right)}
\newcommand{\Tr}[1]{\mathrm{Tr}\left(#1\right)}
\newcommand{\relCh}[2]{\mathrm{Ch}_{#1}\left(#2\right)}
\newcommand{\relTr}[2]{\mathrm{Tr}_{#1}\left(#2\right)}
\newcommand{\tr}{\mathrm{tr}}
\newcommand{\Aut}[1]{\mathrm{Aut}\left(#1\right)}

\newcommand{\ChFam}[2]{\mathrm{Ch}\left(#1;#2\right)}

\newcommand{\conv}[1]{\mathrm{conv}\left(#1\right)}

\newcommand{\Rad}[1]{\mathrm{Rad}\left(#1\right)}

\title[Spectral gap and character limits]{Spectral gap and character limits in \\arithmetic groups}

\author{Arie Levit, Raz Slutsky and Itamar Vigdorovich}

\begin{document}

%\date{}

\maketitle

\begin{abstract}
We establish vanishing results for limits of characters in various discrete groups, most notably irreducible lattices in higher rank semisimple Lie groups. 
As an application, we show that any sequence of finite-dimensional representations converges to the regular representation in the Fell topology. 
%Convergence in the Fell topology of   sequences of finite-dimensional representations to the regular representation follows. 
We achieve this  by studying the geometry of the simplex of traces of discrete groups having Kazhdan's property (T) or its relative generalizations.

%We study the geometry of the simplex of traces on discrete groups with property $\mathrm{T}(\mathcal{R})$ which is a relative  extension of  Kazhdan's property (T). This information is used to establish vanishing results for sequences of characters of irreducible lattices in higher rank semisimple Lie groups.

%\begin{comment}
%We establish vanishing results for sequences of characters of various discrete groups, most notably irreducible lattices in higher rank semisimple Lie groups. This is achieved  by studying the geometry of the simplex of traces of discrete groups having Kazhdan's property (T) or its relative generalizations.
%\end{comment}
%develop notions of spectral gaps of traces and study the geometry of 
% We recover a non-commutative analogue of the Glasner-Weiss theorem, and provide some generalizations for the case where the group does not have property (T).
\end{abstract}

\section{Introduction}

Let $\Gamma$ be a countable group. A \emph{trace} on $\Gamma$ is a positive definite, conjugation-invariant function $\varphi : \Gamma \to \mathbb{C}$ normalized so that $\varphi(e_\Gamma) = 1$.
The space $\Tr{\Gamma}$ of all traces on the group $\Gamma$ is convex and compact   in the topology of pointwise convergence. 
The extreme points of $\Tr{\Gamma}$ are called \emph{characters}. The space of all characters is denoted by $\Ch{\Gamma}$. In general $\Ch{\Gamma}$  need not be compact.

Characters play an important role in harmonic analysis.
For example, the space of characters of the group $\Gamma$ coincides with its Pontryagin dual  when $\Gamma$ is abelian and with the set of equivalence classes of its irreducible representations when $\Gamma$ is finite.
%Furthermore, traces on $\Gamma$ stand in one-to-one correspondence with tracial states on the maximal \csalg  associated to $\Gamma$. The set of traces on such algebras was recognised as an important invariant for decades \cite{rordam2002classification}.
Recently, the study of traces on discrete groups has led to remarkable applications towards understanding the group structure, its dynamics, stability, rigidity and more, see e.g. \cite{bekka2007operator, boutonnet2021stationary, bbhp2022charmenability,hadwin2018stability, levit2022characters,lavi2023characters}. 

We study the character theory of higher rank lattices in semisimple groups. Our main result concerns the asymptotic behaviour of sequences of such characters.

\begin{thm}\label{thm intro:higher rank lattice vanishing of characters}
Let $G$ be a semisimple Lie group with finite center. Assume that $\mathrm{rank}_\mathbb{R}(G) \ge 2$ and that some non-compact simple factor of $G$ has property (T). Let $\Gamma$ be an irreducible lattice in the Lie group $G$. If $\varphi_n \in \Ch{\Gamma}$ is any sequence of pairwise distinct characters then $\varphi_n(\gamma)\to 0$ for all elements $\gamma \in \Gamma \setminus \mathrm{Z}(\Gamma)$.
\end{thm}
 In particular, if the lattice $\Gamma$   is center-free then any such sequence of characters $\varphi_n \in \Ch{\Gamma}$  converges pointwise to the Dirac character $\delta_e \in \Ch{\Gamma}$. 

Our proof of Theorem \ref{thm intro:higher rank lattice vanishing of characters} depends on a more general principle (Theorem \ref{thm:general-statement-for-vanishing-of-characters}). Here are some additional situations where this principle applies, and we  establish asymptotic vanishing results for characters. The full details are provided in \S\ref{sec:character-limits}.
 \begin{enumerate}\label{more_general_examples}
     \item 
     \label{example:global fields}
     Higher rank $S$-arithmetic groups over global fields, provided that the enveloping semisimple group admits a non-compact simple factor with property (T);  see Theorem \ref{thm:generalization to local fields}.
     \item 
     \label{example:SL2}
     The linear groups $\mathrm{SL}_2(S^{-1} \mathcal{O}_k )$ where $k$ is an algebraic number field with ring of integers $\mathcal{O}_k$ and $S \subset \mathcal{O}_k$ is a finite (possibly empty) multiplicative subset such that the localization $S^{-1}\mathcal{O}_k$ has infinitely many units; see Theorem \ref{thm:example 2}.
      \item
      \label{example:congruence}
      In the absence of Kazhdan's property (T) we may rely on Selberg's property instead. We establish convergence of congruence characters, i.e. characters factoring through a congruence quotient;  see Theorem \ref{thm:higher-rank-congruence-limit}.
     \item 
     \label{example:non semisimple}
     Lattices in certain non-semisimple Lie groups, e.g. the group  $\mathrm{GL}_n(\Z)\ltimes \Z^n$ with $n \ge 3$; see Theorem \ref{thm:general arithmetic}. 
 \end{enumerate}

\subsection*{Relation to invariant random subgroups}
An important connection between characters and dynamics is provided by the following notion coined in \cite{abert2014kesten}. An \emph{invariant random subgroup}  of the group $\Gamma$ is a $\Gamma$-invariant Borel probability measure on the Chabauty space of its subgroups.
%The space of all such measures is denoted by $\operatorname{IRS}(\Gamma)$. 
The key observation is that given an invariant random subgroup $\mu$ of the group $\Gamma$ one can construct a trace $\varphi_\mu \in \Tr{\Gamma}$ by setting $\varphi_\mu(\gamma) = \mu(\{ H \leq \Gamma \; | \;\gamma \in H \})$ for all elements $\gamma \in \Gamma$. 
%This  relates the set of traces to the set of invariant random subgroups, and hence to   dynamics. 

A major work dealing with invariant random subgroups of lattices in higher rank semisimple Lie groups is \cite{abert2017growth}.  Many of the striking applications of that work are obtained by showing that every sequence of pairwise distinct irreducible invariant random subgroups of such a lattice converges to a measure supported on subgroups of the center. Theorem \ref{thm intro:higher rank lattice vanishing of characters} is an analogue of this phenomenon for characters.

We mention in passing that the special case of the convergence result from \cite{abert2017growth} for subgroups of a given lattice can be inferred from our Theorem \ref{thm intro:higher rank lattice vanishing of characters} via a careful analysis of traces associated to finite-index subgroups.

% $\varphi_n$ are assumed to be traces, but not necessarily characters. Then the statement remains true provided that the character components appearing in both $\varphi_n$ and $\varphi_m$ become negligible as $n$ or $m$ go to infinity, see. This is the case, for example, for the traces $\varphi_n$ corresponding to the action of $\Gamma$ on $\Gamma/\Gamma_n$
% \[
%     \varphi_n(\gamma)=\frac{1}{[\Gamma:\Gamma_n]} | \{ x\in \Gamma/\Gamma_n : \gamma.x=x \}  |,
% \]
% where $\Gamma_n$ are any sequence of  strictly descending finite-index subgroups $\Gamma_n\leq \Gamma$. We deduce that the invariant random subgroups \[\mu _n= \frac{1}{[\Gamma:\Gamma_n]}\sum_{\gamma\in \Gamma/\Gamma_n} \delta_{\gamma\Gamma_n \gamma^{-1}}\]
% must converge to the invariant random subgroup $\mu$ supported on the center of $\Gamma$. 
% This is equivalent to the Benjamini-Schramm convergence of the Schreier graphs of $\Gamma/\Gamma_n$ to the Cayley graph of $\Gamma$ w.r.t any fixed choice of generators. This partially recovers results from \cite{abert2017growth}.

\subsection*{Character bounds of finite groups}
Note that characters of quotients lift to characters of the group itself. 
Here is an application. Consider any sequence of non-trivial characters $\varphi_i \in \mathrm{SL}_n(\mathbb{F}_{p_i})$ where $p_i \in \mathbb{N}$ are distinct primes and $n \ge 3$. A consequence of Theorem \ref{thm intro:higher rank lattice vanishing of characters} is that  $\varphi_i([\gamma]_{p_i}) \to 0$ for any fixed non-central element $\gamma \in \mathrm{SL}_n(\mathbb{Z})$. Here $[\gamma]_{p_i}$ denotes the entry-wise reduction of the matrix $\gamma$ modulo the prime $p_i$. 
More generally, given a matrix $\gamma$ defined over a ring $R$ and an ideal $I \lhd R$, we let $[\gamma]_I$ denote the matrix obtained by reducing modulo $I$ entry-wise.  We deduce the following statement on characters of   simple (as well as non-simple) finite groups of Lie type; see p. \pageref{sec:finite groups} for the exact definitions.

%As a consequence,  we obtain vanishing results for   characters of sequences of finite simple groups of Lie type.
%fix a non-central element $\gamma \in \mathrm{SL}_n(\Z), n \geq 3$ and choose a sequence $\varphi_p$ of faithful characters on $\mathrm{SL}_n(\mathbb{F}_p)$. Then  implies that $\varphi_p(g) \rightarrow 0$. 

\begin{cor}
\label{cor intro: character limits}
Let $G$ be a Chevalley group of rank at least two\footnote{The Chevalley group $G$ can be one of the classical Chevalley groups $A_n, B_n, C_n$ or $D_n$ for some $n \ge 2$ or one of the exceptional Chevalley groups $E_6, E_7, E_8, F_4$ or $G_2$.}. Let $K$ be a global field and $\mathcal{O}_K$ its ring of integers. Then for any  non-central element $\gamma\in G(\mathcal{O}_K)$ and for any $\varepsilon>0$  the   bound
\begin{equation}
    |\varphi([\gamma]_I)|<\varepsilon
\end{equation}
holds for all but finitely many ideals $I$ of $\mathcal{O}_K$ and for any faithful\footnote{A character $\varphi$ is faithful if $\varphi(g) \neq 1$ for every non-trivial element $g$. If the ideal $I$ is maximal then the group $G(\mathcal{O}_K/I)$ is simple so that  any non-trivial character is faithful.} character $\varphi$ of the finite group $G(\mathcal{O}_K/I)$. 
\end{cor}

Quantitative bounds on characters of finite groups are called character bounds. These  have been extensively studied, see e.g. \cite{gluck1993character,gluck1997characters,larsen2011waring,bezrukavnikov2018character}. Exponential character bounds have found extraordinary applications, most notably in the proof of the Ore conjecture \cite{liebeck2010ore}.
While our statement is not quantitative,   it does provide a relatively simple framework for understanding why such  bounds can be expected. Moreover, character bounds   have been mostly considered for finite \emph{simple} groups. To the best of our knowledge,
Corollary \ref{cor intro: character limits} is novel for certain non-simple groups, such as   $\mathrm{SL}_d(\Z/p^i\Z)$ with  $p$ prime. See also Theorem \ref{thm:semidirect product} for character bounds of    $\mathrm{SL}_d(\F_{p_i})\ltimes \F_{p_i}^d$ provided $d \ge 3$ and with $p_i$ prime.

\subsection*{Convergence of representations}

Convergence of  traces implies convergence of the corresponding representations with respect to the Fell topology; see  Lemma \ref{lem:taking character to its GNS is Fell continuous}. Therefore Theorem \ref{thm intro:higher rank lattice vanishing of characters} on character limits implies  convergence of representations to the regular representation.

\begin{thm}\label{thm intro: convergence of reps}
    Let $\Gamma$ be an irreducible lattice as in Theorem \ref{thm intro:higher rank lattice vanishing of characters} and  with trivial center. Then  
    any sequence of pairwise non quasi-equivalent\footnote{Two finite-dimensional unitary representations $\pi_1$ and $\pi_2$ are called \emph{quasi-equivalent} if they have the same set of irreducible subrepresentations.} finite-dimensional unitary representations of the group $\Gamma$ converges to the regular representation $\lambda_\Gamma$ in the Fell topology. 
\end{thm}

%$\aleph_0 \pi_1$ is equivalent to $\aleph_0 \pi_2$. 
%an irreducible representation $\rho$ is contained in $\pi_1$ if and only if it is contained in $\pi_2$.
Theorem \ref{thm intro: convergence of reps} applies in additional and more  general situations, such as the ones mentioned following   Theorem \ref{thm intro:higher rank lattice vanishing of characters} on page \pageref{more_general_examples}.

%A    statement  in the spirit of Theorem \ref{thm intro:higher-rank-congruence-limit} for   representations factoring through a congruence subgroup is applicable   here as well.

Recall that the Fell topology is not Hausdorff. In particular, limits are not unique. Relying on \cite{bader2021charmenability} it can be shown that \emph{any} limit of a sequence of pairwise disjoint finite-dimensional unitary representations of an irreducible lattice $\Gamma$ as in Theorem \ref{thm intro: convergence of reps} weakly contains the regular representation $\lambda_\Gamma$ (see Proposition \ref{prop:limits of finite dim}). We do not know whether all such limits must, in fact, be weakly equivalent to the regular representation $\lambda_\Gamma$.

%We deduce the following application of Theorem \ref{thm intro: convergence of reps}.
\begin{cor}
\label{cor:for descending stuff}   
    Let $\Gamma$ be an irreducible lattice as in Theorem \ref{thm intro:higher rank lattice vanishing of characters} and with trivial center. Let $\Gamma_n \le \Gamma$ be a strictly descending chain of finite-index subgroups. Then the sequence of quasi-regular representations $L^2(\Gamma/\Gamma_n)$ converges to the regular representation $\lambda_\Gamma$ in the Fell topology.
\end{cor}

Corollary \ref{cor:for descending stuff}  can also be deduced   as a consequence of the Benjamini--Schramm convergence of the sequence of Schreier graphs of $\Gamma/\Gamma_n$ to the Cayley graph of the group $\Gamma$, which in turn follows from \cite{abert2017growth}. Note that the quasi-regular representations $L^2(G/\Gamma_n)$ of the semisimple Lie group $G$  converge to its regular representation $\lambda_G$,  by the continuity of induction.

\subsection*{Spectral gap and the   simplex of traces}

%Let us describe the general strategy for the simpler case, when $G$ has property (T).
Let $\Gamma$ be a countable group. The compact and   convex space of traces on the  group $\Gamma$ is   a metrizable \emph{Choquet simplex},\footnote{We will sometimes use the shorter terminology \emph{simplex} to refer to Choquet simplices. We refer to \cite{phelps2001lectures} for the general theory of simplices.}   i.e. every trace $\varphi \in \Tr{\Gamma}$ can be expressed as  the barycenter of a \emph{unique} Borel probability measure supported on the set of its extreme points $\Ch{\Gamma}$. A  Choquet simplex is called \emph{Bauer} if its set of  extreme points is closed (and hence compact).

Our strategy to prove Theorem \ref{thm intro:higher rank lattice vanishing of characters} is to exploit the  geometric structure of the simplex $\Tr{\Gamma}$. First, we show that the simplex $\Tr{\Gamma}$ is Bauer.  Next,  we show that every   non-central character is an   isolated point of the subset of extreme points of this Bauer simplex. These two facts immediately imply that every sequence of pairwise distinct characters of the group $\Gamma$ must converge to a character supported on the center.

To  pinpoint closed subsets of the space of characters  $\Ch{\Gamma}$  we make use of  the notion of property $\mathrm{T}(\mathcal{R})$. Let  $\mathcal{R}$ be a set of equivalence classes of unitary  representations of the group $\Gamma$. For each unitary representation $\pi \in \mathcal{R}$ let  $\pi_0$ denote its subrepresentation obtained by taking the orthogonal complement to the subspace of invariant vectors. The group $\Gamma$ has \emph{property $\mathrm{T}(\mathcal{R})$} if the trivial representation $1$ is an isolated point of the set 
%\begin{equation}
$\{ \pi_0 : \pi \in \mathcal{R} \} \cup \{ 1 \}    $
%\end{equation}
with respect to the Fell topology. In particular, if the family $\mathcal{R}$ consists of all equivalence classes of irreducible unitary representations then  property $\mathrm{T}(\mathcal{R})$ is the same as  Kazdhan's property (T). This generalized property was first introduced in \cite{lubotzky1989variants}. See also \S\ref{sec:convergence of traces}  below for further discussion. 
Denote
\begin{equation}
\ChFam{\Gamma}{\mathcal{R}} = \{ \varphi \in \Ch{\Gamma} \: : \: \pi_\varphi \in \mathcal{R} \}
\end{equation}
where $\pi_\varphi$ is the unitary representation associated to the character $\varphi$ via the GNS construction, see  \S\ref{sec:traces and characters}.
 When discussing the subset $\ChFam{\Gamma}{\mathcal{R}}$  we will be making a very mild assumption, namely that the family $\mathcal{R}$  is so called \emph{hereditary}, a technical notion defined in  \S\ref{sec:convergence of traces}.
The first ingredient towards Theorem \ref{thm intro:higher rank lattice vanishing of characters} and its generalizations is the following.

%First, we recall that given a trace $\varphi$ on a group $\Gamma$, there exists a unitary representation $\pi_\varphi$ associated with it through the GNS construction.
%This will be reviewed more accurately in Section \ref{sec:traces and characters}. This connection between traces and the representation theory of the group is the main bridge for developing notions of spectral gaps.
%We denote by $\Ch{\Gamma, \mathcal{R}}$ the set of extremal traces $\varphi$ such that $\pi_\varphi \in \mathcal{R}$.

\begin{thm}
\label{thm intro:Bauer-simplex-face-characters}
    Assume that the countable group $\Gamma$ has property $\mathrm{T}(\mathcal{R})$. Then $\overline{\ChFam{\Gamma}{\mathcal{R}}} \subset \Ch{\Gamma}$. Moreover, the closed convex hull %$
    $
    %\hspace{1mm}
    \overline{\mathrm{conv} \, \Ch{\Gamma;\mathcal{R}}}$ is a closed face of the simplex $\Tr{\Gamma}$ and is a Bauer simplex.
\end{thm}
The second part of Theorem \ref{thm intro:Bauer-simplex-face-characters} follows from its first part purely abstractly. Specializing to the case where the family $\mathcal{R}$ consists of all   equivalence classes of separable unitary representations  we immediately obtain the following.

\begin{cor}
\label{cor intro: Kazhdan group has Bauer simplex}
If the countable group $\Gamma$ has  Kazhdan's property (T) 
then $\Tr{\Gamma}$ is a Bauer simplex.
\end{cor}

This statement is inspired by   the well-known result of Glasner and Weiss \cite{glasner1997kazhdan} which deals with the geometry of the simplex of invariant probability measures for a $\Gamma$-action on a compact space.
We remark that Corollary \ref{cor intro: Kazhdan group has Bauer simplex} can specifically be deduced from the recent deep work of Kennedy and Shamovich in which the authors develop a non-commutative Choquet theory and extend the result of \cite{glasner1997kazhdan} to that realm \cite[\S 14]{kennedy2022noncommutative}.

Note that there are non-Kazhdan groups whose simplex of traces is also Bauer. For instance, if $\Gamma$ is discrete and abelian, then $\Ch{\Gamma}$ coincides with its Pontryagin dual $\widehat{\Gamma}$ which is compact. On the other hand, the simplex of  traces of many nilpotent groups, among them  the discrete Heisenberg group, is not Bauer \cite[Example 1.9]{bader-iti2022charmenability}. We do not know of any precise characterization of the family of groups whose simplex of traces is Bauer.

The second ingredient of the proof of Theorem \ref{thm intro:higher rank lattice vanishing of characters} deals with the geometry of the subset of finite-dimensional traces.

\begin{prop}
\label{prop:finite dim is not an accumulation point}
 Assume that the countable group $\Gamma$ has property $\mathrm{T}(\mathcal{R})$. Then any finite-dimensional character $\varphi \in \overline{\ChFam{\Gamma}{\mathcal{R}}}$ is an isolated point of the subset $\overline{\ChFam{\Gamma}{\mathcal{R}}}$.
 %Any finite-dimensional representation $\rho \in \overline{\mathcal{R}}$ satisfies $\rho \in \mathcal{R}$. Furthermore any irreducible finite-dimensional representation $\rho \in \overline{\mathcal{R}} \cap \widehat{\Gamma}$  is   isolated  in the subset $\overline{\mathcal{R}} \cap \widehat{\Gamma}$ in the Fell topology.
\end{prop}

Specializing once again  to the case where the family $\mathcal{R}$ consists of all equivalence classes of separable unitary representations, we immediately obtain  the following.

\begin{cor}
\label{cor intro: Kazhdan group has isolated fd chars}
If the countable group $\Gamma$ has Kazhdan's property $(\mathrm{T})$ then any finite-dimensional character of $\Gamma$ is an isolated point of  $\Ch{\Gamma}$.
\end{cor}

%For instance, the trivial character $1 \in \Ch{\Gamma}$ is an isolated point of $\Ch{\Gamma}$ provided the discrete group $\Gamma$ has property (T). 
This is reminiscent of the definition of property (T) saying that finite-dimensional representations are isolated in the Fell topology. In fact, it is not hard to deduce Corollary \ref{cor intro: Kazhdan group has isolated fd chars} directly from that definition. The converse direction of Corollary \ref{cor intro: Kazhdan group has isolated fd chars} is false however,  as can be seen in some of  the examples following Theorem \ref{thm intro:higher rank lattice vanishing of characters}. 

\subsection*{On relative traces}

Finally, we discuss a relative version of Theorem \ref{cor intro: Kazhdan group has Bauer simplex}.

\begin{thm}
\label{thm intro: characteristic characters is closed}
Let $\Lambda$ be a countable group admitting an action $\theta:\Lambda \to \Aut{\Gamma}$ on the countable group $\Gamma$. Assume that $\Lambda$ has Kazhdan's property (T) and that $\mathrm{Inn}(\Gamma) \le \theta(\Lambda)$. Then 
\begin{itemize}
    \item the space  $\relTr{\Lambda}{\Gamma}$ consisting of all $\Lambda$-invariant traces on the group $\Gamma$ is a Bauer simplex, and
    \item finite-dimensional extremal $\Lambda$-invariant traces are isolated points of the extreme point set  $\relCh{\Lambda}{\Gamma} = \partial \relTr{\Lambda}{\Gamma}$.
\end{itemize}
\end{thm}

To illustrate Theorem \ref{thm intro: characteristic characters is closed} consider the free abelian group $\Gamma=\Z^d$ for some $d \ge 3$. Its group of automorphisms   $\Lambda=\mathrm{Aut}(\Gamma) \cong \mathrm{GL}_d(\Z)$ has property (T).  It is an immediate consequence  of Theorem \ref{thm intro: characteristic characters is closed} that 
 any sequence of pairwise distinct ergodic $\mathrm{GL}_d(\Z)$-invariant measures on the torus $\T^d$ converges to the Haar measure in the weak-$*$ topology, see Corollary \ref{cor:torus-measures-limit}. This fact is well  known to experts  for all $d \ge 2$.  In any case, combining this information with vanishing results for  arithmetic groups (not necessarily semisimple ones, see Theorem \ref{thm:general arithmetic}) gives   the asymptotic vanishing of sequences of characters on $\mathrm{GL}_d(\Z)\ltimes \Z^d$ (Theorem \ref{thm:semidirect product}).

Next, consider   the free group $\Gamma=F_d$    of rank $d\geq 2$. Let   $\Lambda=\Aut{F_d}$ be its group of   automorphisms. This can be   viewed as a non-commutative analogue of the   discussion in the previous paragraph. Unlike the simplex $\Tr{\Z^d}$, the space of traces $\Tr{F_d}$   is as far as possible  from being a Bauer simplex --- it is a Poulsen simplex, meaning that  $\overline{\Ch{F_d}} = \Tr{F_d}$ \cite{iti-raz-Orovitz}. However, the celebrated results of \cite{kaluba2019aut,kaluba2021property,nitsche2020computer} say that the group of automorphisms $\mathrm{Aut}(F_d)$ has property (T) provided  $d \ge 4$.\footnote {The group $\mathrm{Aut}(F_d)$ is not Kazhdan for $1 \le d \le 3$.} So Theorem \ref{thm intro: characteristic characters is closed} can be brought to bear.
Assuming  $d \ge 4$ we conclude that    the space $\relTr{\mathrm{Aut}(F_d)}{F_d}$ is a Bauer simplex  after all.

$\operatorname{Aut}(F_d)$-invariant positive definite functions on the free group $F_d$ have recently attracted significant interest  with connection to word measures on groups and the Wiegold conjecture,  see \cite{collins2019automorphism} for a survey.
Such positive definite functions are in fact $\operatorname{Out}(F_d)$-invariant \emph{traces}. The above result  is  an initial step towards understanding their structure.

\subsection*{Structure of the paper}
In \S\ref{sec:traces and characters} and \S\ref{sec:conjugation representation} we discuss general properties of traces and characters. In \S\ref{sec:Traces and spectral gap} we study various notions of spectral gap of traces. These notions correspond to closed conditions on the space of all traces. This analysis is used in \S\ref{sec:convergence of traces} to discuss the geometry of the simplex of traces on groups with property $\mathrm{T}(\mathcal{R})$. In  particular we tackle Theorem \ref{thm intro:Bauer-simplex-face-characters} and Proposition \ref{prop:finite dim is not an accumulation point}. Applications to character and representation limits are discussed in \S\ref{sec:character-limits}. This is where our main results, Theorems \ref{thm intro:higher rank lattice vanishing of characters}, \ref{thm intro: convergence of reps} and their generalizations are proved. Lastly in \S\ref{sec:rel-traces} we discuss relative characters and prove  Theorem \ref{thm intro: characteristic characters is closed}.

\subsection*{Acknowledgements} A.L was supported by ISF grant 1788/22 and BSF grant No. 2022105. R.S and I.V were supported by ERC advanced grant No. 101052954.

\section{Traces and characters}
\label{sec:traces and characters}
We start with some preliminaries regarding traces and characters on groups. For a comprehensive overview of this theory we refer to \cite[Chapter 11]{bekka2020unitary}.

Let $\Gamma$ be a countable discrete group. A \emph{trace} on $\Gamma$ is a function $\varphi : \Gamma \to \mathbb{C}$ satisfying the following conditions:
\begin{enumerate}
    \item $\varphi$ is \emph{positive definite}: $\sum_{i=1}^{n} {\alpha}_i \overline{\alpha}_j\varphi(\gamma_j^{-1}\gamma_i)\geq 0$ for any $n\in \N$ and any choice of group elements $\gamma_1,...,\gamma_n\in \Gamma$ and scalars $\alpha_1,...,\alpha_n\in \C$.
    \item $\varphi$ is \emph{conjugation-invariant}: $\varphi(x^{-1}\gamma x)=\varphi(\gamma)$ for all group elements $\gamma,x\in \Gamma$.
    \item $\varphi$ is \emph{normalized}: $\varphi(e_\Gamma) = 1$.
\end{enumerate}

Let $\Tr{\Gamma}$ denote the space of all traces on the group $\Gamma$ equipped with the topology of pointwise convergence. It is not hard to see that $|\varphi(\gamma)| \le \varphi(e_\Gamma) = 1$ holds for all group elements $\gamma \in \Gamma$. As such, the set $\Tr{\Gamma}$ can be regarded as a convex weak-$*$ compact\footnote{The weak-$*$ topology  coincides with the topology of pointwise convergence on $l^\infty(\Gamma)$.} subset of $l^\infty(\Gamma)$. In particular, any trace $\varphi$ on $\Gamma$ extends uniquely to a linear functional on $\ell^1(\Gamma)$. We shall use this extension without further mention.  

Let  $\Ch{\Gamma}$ denote the subset of extreme points of the convex set   $\Tr{\Gamma}$. Elements of $\Ch{\Gamma}$ are called\footnote{Some authors refer to all traces as \enquote{characters} and to extreme traces as \enquote{indecomposable characters}.} \emph{characters} of the group $\Gamma$.

Consider some trace $\varphi \in \Tr{\Gamma}$. Associated to the trace $\varphi$ is the \emph{GNS (Gelfand--Naimark--Segal)  data} $(\pi_\varphi,\rho_\varphi,\mathcal{H}_\varphi,v_\varphi)$. Here $\mathcal{H}_\varphi$ is a Hilbert space and  $\pi_\varphi,\rho_\varphi  : \Gamma \to \mathrm{U}(\mathcal{H}_\varphi)$   is a pair of unitary representations. 
Moreover $v_\varphi \in \mathcal{H}_\varphi$ is a cyclic unit vector   for $\pi_\varphi$ as well as for $\rho_\varphi$ satisfying $\pi_\varphi(\gamma) v_\varphi = \rho_\varphi(\gamma^{-1}) v_\varphi$ for all elements $\gamma \in \Gamma$. 
The two representations $\pi_\varphi$ and $\rho_\varphi$ commute, namely $\pi_\varphi(\gamma_1) \rho_\varphi(\gamma_2) = \rho_\varphi(\gamma_2) \pi_\varphi(\gamma_1)$ for any pair of elements $\gamma_1,\gamma_2 \in \Gamma$.  Lastly 
\begin{equation}
\varphi(\gamma) = \left<\pi_\varphi(\gamma)v_\varphi,v_\varphi\right> = \left<\rho_\varphi(\gamma^{-1})v_\varphi,v_\varphi\right> \quad \forall \gamma \in \Gamma.
\end{equation}
These requirements determine the GNS data uniquely up to an isomorphism. As   the group $\Gamma$ is countable   the Hilbert space $\mathcal{H}_\varphi$ is separable for every   $\varphi \in \Tr{\Gamma}$.

Associated to the trace $\varphi \in \Tr{\Gamma}$ is the von Neumann algebra $M_\varphi$ generated by the operators $\pi_\varphi(\gamma)$ for all $\gamma \in \Gamma$. In other words $M_\varphi = \pi_\varphi(\Gamma)''\subseteq\mathrm{B}(\mathcal{H}_\varphi)$ by the von Neumann bicommutant theorem. The trace $\varphi$ is a character if and only if the von Neumann algebra $M_\varphi$ is a \emph{factor}, namely $\mathrm{Z}(M_\varphi) = \mathbb{C}$. The group $\Gamma$ admits an \emph{adjoint representation} on the von Neumann algebra $M_\varphi$ given by
\begin{equation}
\mathrm{Ad}_\gamma(x) = \pi_\varphi(\gamma) x \pi_\varphi(\gamma^{-1}) \quad \forall \gamma \in \Gamma, x \in M_\varphi.
\end{equation}

\subsection*{Dominated traces}
Let $\varphi, \psi \in \Tr{\Gamma}$ be a pair of traces. We say that the trace $\psi$ is \emph{dominated} by the trace $\varphi$ and write $\psi \le \varphi$ if  there is some $0 < \alpha \le 1$ such that the function $\varphi - \alpha \psi$ is positive definite. There is a surjective mapping from the set of non-zero operators $T \in \mathrm{Z}(M_\varphi)$ satisfying $0 < T \le 1$ to the set of traces $\psi$ dominated by the trace $\varphi$. This mapping takes such an element $T$ to the trace $\varphi_T$ given by
\begin{equation}
\label{eq:varphi_T}
    \varphi_T(\gamma) = \frac{1}{\|T^\frac{1}{2} v_\varphi\|^2} \left<\pi_\varphi(\gamma) T^{\frac{1}{2}} v_\varphi, T^{\frac{1}{2}} v_\varphi \right> \quad \forall \gamma \in \Gamma.
\end{equation}
Moreover $\varphi_{T_1} = \varphi_{T_2}$ for a pair of elements $T_1, T_2$ as above if and only if $T_1 = \alpha T_2$ for some $\alpha > 0 $. See \cite[Lemma 11.C.2]{bekka2020unitary}.

% For each  element $T \in \mathrm{Z}(M_\varphi)$ satisfying $0 \le T \le 1$ the trace $\varphi_T \in \Tr{\Gamma}$ given by the formula

% satisfies $\varphi_T \le \varphi$. Indeed, every trace $\psi$ dominated by the trace $\varphi$ is of the form $\psi = \varphi_T$ for some operator $T \in \mathrm{Z}(M_\varphi) $ with $0 \le T \le 1$  

%The set of traces $\psi$ dominated by the trace $\varphi$  corresponds bijectively to the set of all   The correspondence takes such an  element $T$ to the trace $\varphi_T$ given by 

\subsection*{The trivial character} The trivial character $1 \in \Ch{\Gamma}$ is defined by $1(\gamma) = 1$ for all group elements $\gamma \in \Gamma$.

\begin{prop}
\label{prop:condition for pi to admit invariant vectors}
Let $\varphi \in \Tr{\Gamma}$ be a trace. The unitary representation $\pi_\varphi$   has invariant vectors if and only if $1 \le \varphi$.
\end{prop}
\begin{proof}
% \marginpar{this is a "basic proof"}
% Write $\mathcal{H}_\varphi = \mathcal{H}_0 \oplus \mathcal{H}_1$ where $\mathcal{H}_0$ consists of $\pi_\varphi$-invariant vectors and $\mathcal{H}_1 = \mathcal{H}_0^\perp$. Likewise write $v_\varphi = v_0 + v_1$ where $v_0 \in \mathcal{H}_0$ and $v_1\in\mathcal{H}_1$. We get that $\varphi = \varphi_0 + \varphi_1$ where $\varphi_i = \left<\pi_\varphi(\gamma)v_i,v_i\right>$ so that $\varphi_0 = \|v_0\|^2$. If the subspace $\mathcal{H}_0$ is non-trivial we conclude that $1 \le \varphi$.

% Conversely suppose that $\varphi = \alpha + (1-\alpha) \psi$ for some $\alpha > 0$ and some $\psi \in \Tr{\Gamma}$. The uniqueness of the GNS construction implies that $\pi_\varphi = 1 \oplus \pi_\psi$ and $v_\varphi = v_0 \oplus v_\psi$ where $v_0$ is a $\pi_\varphi$-invariant vector.
First, assume that the unitary representation $\pi_\varphi$ admits invariant vectors. The orthogonal projection $T$ to the subspace of invariant vectors is non-zero and satisfies $ 0 \le T \le 1$ as well as $T \in \mathrm{Z}(M_\varphi)$. Therefore 
 $\varphi_T \le \varphi$ and $ \varphi_T = 1 $   by the formula in Equation (\ref{eq:varphi_T}).
In the converse direction, assume that $\varphi_T = 1$ for some non-zero $T \in \mathrm{Z}(M_\varphi)$ satisfying $0 \le T \le 1$. Then $T^{\frac{1}{2}}v_\varphi \in \mathcal{H}_\varphi$ is a non-zero invariant vector.
\end{proof}

If $\varphi \in \Ch{\Gamma}$ and $\psi \in \Tr{\Gamma}$ satisfy $\psi \le \varphi$ then certainly $\varphi = \psi$. Combined with Proposition \ref{prop:condition for pi to admit invariant vectors} this  gives the following.

\begin{cor}
\label{cor:condition for a character to be trivial}
Let $\varphi \in \Ch{\Gamma}$ be a character. The unitary representation $\pi_\varphi$   has invariant vectors if and only if $\varphi $ is the trivial character $1$.
\end{cor}

Note that the statement of Corollary \ref{cor:condition for a character to be trivial} is equivalent to saying that   the unitary representation $\pi_\varphi$ is trivial.

The above statements hold true in the same way if the unitary representation $\pi_\varphi$ is replaced by the unitary representation $\rho_\varphi$.

\subsection*{Finite-dimensional characters}

Let $\mathcal{H}$ be a Hilbert space with conjugate Hilbert space $\overline{\mathcal{H}}$. The Hilbertian tensor product $\mathcal{H} \otimes \overline{\mathcal{H}}$ is naturally isomorphic to the Hilbert space $\mathrm{HS}(\mathcal{H})$ of all Hilbert--Schmidt operators acting on $\mathcal{H}$.
The space $\mathrm{HS}(\mathcal{H})$ with its Hilbert--Schmidt norm admits a dense subspace $\mathrm{FR}(\mathcal{H})$ consisting of all finite rank operators. Note that the identity map $\mathrm{Id}$  belongs to $\mathrm{HS}(\mathcal{H})$ if and only if the Hilbert space $\mathcal{H}$ is finite-dimensional.

Let $\pi$ be a unitary representation acting on the Hilbert space $\mathcal{H}$. Then the unitary representation $\pi \otimes \pi^*$ acts on the tensor product Hilbert space $\mathcal{H} \otimes \overline{\mathcal{H}}$, and this action is equivalent to the conjugation action on   $\mathrm{HS}(\mathcal{H})$.
An invariant vector $T \in \mathrm{HS}(\mathcal{H})$ for the unitary representation $\pi \otimes \pi^*$ is an operator satisfying $\pi(\gamma) T = T \pi(\gamma)$ for all elements $\gamma \in \Gamma$. In other words, such an invariant vector $T$ is  an intertwining operator for the representation $\pi$.
%, i.e.  $T \in \pi'_\varphi$. 

A trace $\varphi \in \Tr{\Gamma}$ is said to be \emph{finite-dimensional} if its GNS unitary representation $\pi_\varphi$ is finite-dimensional.

\begin{lemma}
\label{lemma:condition for invariant vectors for tensor}
Let $\varphi \in \Tr{\Gamma}$ be a trace. The unitary representation $\pi_\varphi \otimes \pi_\varphi^*$ has invariant vectors if and only if  $\psi \le \varphi$ for some finite-dimensional character $\psi \in \Ch{\Gamma}$.
\end{lemma}
\begin{proof}
Assume that there is a finite-dimensional character $\psi \in \Ch{\Gamma}$  given by  $\psi = \varphi_T$ for some non-zero operator $T \in \mathrm{Z}(M_\varphi)$ satisfying $0 \le T \le 1$. 
Consider the closed subspace $\mathcal{L}$ of the Hilbert space $\mathcal{H}_\varphi$ given by
\begin{equation}
\mathcal{L} = \overline{ \mathrm{span}_\mathbb{C} \{ \pi_\varphi(\gamma) v_\psi \: :\: \gamma \in \Gamma \}  }  
\end{equation}
where 
 $v_\psi = T^{\frac{1}{2}} v_\varphi$. The GNS data associated to the character $\psi$ is uniquely determined up to an isomorphism. This implies that the subspace $\mathcal{L}$ is isomorphic to the Hilbert space $\mathcal{H}_\psi$. As such, the closed subspace $\mathcal{L}$ is finite-dimensional. Therefore the orthogonal projection operator $\mathrm{P}$ to the subspace $\mathcal{L}$ belongs  to $\mathrm{HS}(\mathcal{H}_\varphi)$. Moreover   $\mathrm{P}$ is an invariant vector for the unitary representation $\pi_\varphi \otimes \pi_\varphi^*$ as  the  subspace $\mathcal{L}$ is $ \pi_\varphi(\Gamma)$-invariant.

Conversely, suppose that the unitary representation $\pi_\varphi \otimes \pi^*_\varphi$ admits invariant vectors. It follows that the unitary representation $\pi_\varphi$ admits a finite-dimensional sub-representation $\pi_0$  \cite[A.1.12]{bekka2008kazhdan}.
Let $\mathrm{P}$ be the orthogonal projection to the sum of all the sub-representations of $\pi_\varphi$ isomorphic to $\pi_0$. Then $\mathrm{P} \in \pi_\varphi(\Gamma)'$. In addition $\mathrm{P} \in \pi_\varphi(\Gamma)''$ as the von Neumann algebra generated by a direct sum of disjoint representations is the direct sum of the corresponding von Neumann algebras \cite[Proposition 6.B.4]{bekka2020unitary}.
It follows that  $\mathrm{P} \in \mathrm{Z}(M_\varphi)$. Denote $\psi = \varphi_\mathrm{P}$ so that $\psi \le \varphi$. Note that the trace $\psi$ is finite-dimensional.
\end{proof}
Restricting  once again to the case where $\varphi$ is a character, we get the following  useful characterization.
\begin{cor}
Let $\varphi \in \Ch{\Gamma}$ be a character. Then $\varphi$ is finite-dimensional if and only if the representation $\pi_\varphi \otimes \pi^*_\varphi$ has invariant vectors.
\end{cor}

\subsection*{Amenable traces}
A trace $\varphi \in \Tr{\Gamma}$ is called \emph{amenable} if the corresponding unitary representation $\pi_\varphi$ is amenable in the sense of Bekka \cite{bekka1990amenable}, namely the tensor product representation $\pi_\varphi \otimes \pi^*_\varphi$ admits almost-invariant vectors.  This is equivalent to saying that the norm of the positive contracting  operator $(\pi_\varphi\otimes \pi_\varphi^*)(a)$ is $1$ for some choice of  an element $a \in \ell^1(\Gamma)$ such that the support of the element $a^*a$ generates the group $\Gamma$ \cite[G.4.2]{bekka2008kazhdan}. A non-amenable trace cannot dominate any finite-dimensional trace, see Lemma \ref{lemma:condition for invariant vectors for tensor}. In particular, a non-amenable character is  infinite-dimensional.

\section{The conjugation representation}
\label{sec:conjugation representation}

In this section, we turn our attention to the conjugation representation associated to a trace. This representation will be helpful in determining whether the trace in question is in fact a character. 

Let $\Gamma$ be a countable discrete group and let $\varphi \in \Tr{\Gamma}$ be a trace on $\Gamma$.
Consider the unitary representation $c_\varphi:\Gamma\to \mathrm{U}(\mathcal{H}_\varphi)$ given by 
\begin{equation}
c_\varphi(\gamma)=\pi_\varphi(\gamma)\rho_\varphi(\gamma)=\rho_\varphi(\gamma)\pi_\varphi(\gamma) \quad
\forall \gamma \in \Gamma.
\end{equation}
Note that the representation $c_\varphi$ is the restriction of the representation 
\begin{equation}
    \pi_\varphi\times \rho_\varphi:\Gamma\times \Gamma\to \mathrm{U}(\mathcal{H}_\varphi)
\end{equation}
to the diagonal subgroup $\Gamma\cong \Delta \Gamma \leq \Gamma\times \Gamma$.
%Denote $c = \pi \circ \rho = \rho \circ \pi$. 

Let $M_\varphi$ be the von Neumann algebra associated to the trace $\varphi$. Consider the linear map  
\begin{equation}\label{eq:iota}
    \iota_\varphi : M_\varphi\to\mathcal{H}_\varphi,\quad  \iota :  x\mapsto xv_\varphi \quad \forall x \in M_\varphi.
\end{equation} 
The fact that the vector $v_\varphi$ is cyclic for the representation $\pi_\varphi$ implies that the image of $\iota_\varphi$ is dense. Likewise, the fact that the vector $v_\varphi$ is cyclic for the representation $\rho_\varphi$ combined with the fact that $\rho_\varphi(\Gamma) \subset  M'_\varphi$  implies that $\iota_\varphi$ is injective. The map
$\iota_\varphi$ intertwines the adjoint representation of the group $\Gamma$ on the algebra $M_\varphi$ with the representation $c_\varphi$ on the Hilbert space  $\mathcal{H}_\varphi$, namely 
\begin{align}
\begin{split}
\iota_\varphi(\mathrm{Ad}_\gamma(x)) &=  \mathrm{Ad}_\gamma(x)v_\varphi = \left(\pi_\varphi(\gamma) x \pi_\varphi(\gamma^{-1}) \right) v_\varphi= 
  \pi_\varphi(\gamma) x \rho_\varphi(\gamma) v_\varphi = \\ &= \pi_\varphi(\gamma) \rho_\varphi(\gamma) x v_\varphi = c_\varphi(\gamma) x v_\varphi = c_\varphi(\gamma) \iota_\varphi(x)
\end{split}
\end{align}
for all group elements $\gamma \in \Gamma$ and all $x \in M_\varphi$.
If follows that the image of  $\iota_\varphi$ is a $c_\varphi$-invariant subspace of  $\mathcal{H}_\varphi$ and that the center  $\mathrm{Z}(M_\varphi)$ is mapped injectively into the subspace of $c_\varphi$-invariant vectors $\mathcal{H}_\varphi^{c_\varphi}$.

The main goal of the current section is the following result providing information about the image of 
the linear map $\iota_\varphi$.

\begin{prop}
\label{prop:embedding with dense image of H^c into Z(M)}
The embedding   $\iota_\varphi:\mathrm{Z}(M_\varphi)  \hookrightarrow  \mathcal{H}_\varphi^{c_\varphi}$ has dense image.
\end{prop}

The proof of Proposition \ref{prop:embedding with dense image of H^c into Z(M)} will be given below. As was already mentioned, the trace $\varphi$ is a character if and only if the center $\mathrm{Z}(M_\varphi)$ of the von Neumann algebra $M_\varphi$ is one-dimensional. Therefore Proposition \ref{prop:embedding with dense image of H^c into Z(M)} immediately implies the following  fact. While this is known to experts,   we are not aware of any  straightforward reference.

\begin{cor}
\label{cor:c-invariant vectors are one-dimensional iff its a character}
The trace $\varphi$ is a character if and only if $\dim _\mathbb{C} \mathcal{H}_\varphi^{c_\varphi} = 1$.
\end{cor}

The   condition $\dim _\mathbb{C} \mathcal{H}_\varphi^{c_\varphi} = 1$ is equivalent to saying that  $\mathcal{H}_\varphi^{c_\varphi} = \mathrm{span}_\mathbb{C}   v_\varphi $. We generalize this statement to the relative setting in \S\ref{sec:rel-traces}. 

\subsection*{Dense linear subspaces of Hilbert spaces}
The proof of Proposition \ref{prop:embedding with dense image of H^c into Z(M)} will rely on the following general lemma concerning Hilbert spaces.

\begin{lemma}
\label{lem: a lemma on vector spaces}
Let $V$ be a Hilbert space. Let $C$ be a closed   convex   balanced\footnote{Recall that a subset $C$ of a complex vector space is called \emph{balanced} if $aC \subset C$ for all scalars $a$ with $|a| \le 1$.}   subset of $V$ such that the subspace  $A = \bigcup_n nC$ is  dense   in  $V$. Denote by $\Pi$  the nearest point projection from $V$ to $C$.    Then $A \cap W$ is dense in $W$ for any closed subspace $ W \le V$ satisfying $\Pi(W) \subset W$.
\end{lemma}
\begin{proof}
Let $W$ be any closed subspace of $V$ satisfying $\Pi(W) \subset W$. 
Consider any vector  $v \in W \cap (A \cap W)^\perp$. We will show that necessarily $v= 0$. 

We start by showing that $\Pi(v)=0$. Note that the nearest point projection $\Pi(v)$ satisfies
\begin{equation}
    \Pi(v) \in C \cap W \subset A \cap W.
\end{equation}
The properties of the nearest point projection $\Pi(v)$ imply that for any other  vector  $w\in C$     the angle at the point $\Pi(v)$ between the two line segments $\left[\Pi(v), v\right]$ and $\left[ \Pi(v),w\right]$  cannot be acute. In terms of the inner product this means that
\begin{equation}
\label{eq:acute angle}
\Re \left<v-\Pi(v), w-\Pi(v) \right> \le 0 \quad  \forall w\in C.
\end{equation}
Substituting  $w = 0$ and recalling that $v \in (A \cap W)^\perp$ while $\Pi(v) \in A \cap W$  gives $\Re \left<\Pi(v),\Pi(v)\right> \le 0$. Therefore  $\Pi(v) = 0$. As Hilbert spaces are strictly convex, we conclude that 
\begin{equation}
\label{eq:strict stuff}
    \|v\| = \| v - \Pi(v) \| < \|v-w\|
\end{equation}
for all   vectors $  w \in C$ with $w \neq \Pi(v) = 0$.

The linear subspace $A$ is dense in the Hilbert space $V$ by assumption. Therefore there is a sequence of vectors $u_n \in nC$ converging to $v$. Denote $x_n = \frac{1}{n}u_n  $ so that $x_n \in C$ for all $n \in \mathbb{N}$ and  $x_n \to 0$. 
As $u_n \to v$ we have that $\left<u_n,v\right> \to \left<v,v\right>$ as well as $\left<u_n,u_n\right> \to \left<v,v\right>$. Assume towards contradiction that $v \neq 0$. In that case  
\begin{equation}
\label{eq:inner product strictly larger}
\frac{2}{n} \Re \left<u_n,v\right> > \frac{1}{n^2} \left<u_n,u_n\right>   
\end{equation}
for all $n$ sufficiently large. 
On the other hand, for all $n$ we have that
\begin{equation}
\label{eq:inner product open}
\left<v-x_n,v-x_n\right> = \left<v,v\right> - \frac{2}{n}\Re \left<u_n,v\right> + \frac{1}{n^2} \left<u_n,u_n\right>.
%\measuredangle_0(w,x_n) = \measuredangle_0(w,u_n) \to 0 \quad \text{and} \quad \measuredangle_w(0,x_n) \to 0.
\end{equation}
Putting together the two Equations (\ref{eq:inner product strictly larger}) and  (\ref{eq:inner product open})   implies that 
 $\|v\| > \|v - x_n\|  $ for all $n$ sufficiently large. 
It follows from the discussion involving Equation (\ref{eq:strict stuff})  that $x_n = u_n = 0$ for all $n$ sufficiently large. As $u_n \to v$ we arrive at a contradiction to the assumption $v \neq 0$ and deduce that $v = 0$  as required.
\end{proof}

\subsection*{Invariant vectors for the conjugation representation}

Let $(M_\varphi)_1$ denote the subset of the von Neumann algebra $M_\varphi\subset \mathrm{B}(\mathcal{H}_\varphi)$ consisting of all operators with operator norm at most $1$. Identify the von Neumann algebra $M_\varphi$ as a subspace of the Hilbert space $\mathcal{H}_\varphi$ via the map $\iota_\varphi$ introduced in Equation (\ref{eq:iota}). This identification makes $(M_\varphi)_1$   a closed subset of the Hilbert space $\mathcal{H}_\varphi$  by \cite[Corollary 5.2]{peterson2016lecture}.  Moreover the subset $(M_\varphi)_1$ is convex and $c_\varphi$-invariant. We are ready to complete the following proof.

\begin{proof}[Proof of Proposition \ref{prop:embedding with dense image of H^c into Z(M)}]
The desired statement follows immediately as a special case of  Lemma \ref{lem: a lemma on vector spaces}. We let $V$ be the Hilbert space $\mathcal{H}_\varphi$, $W$ be the Hilbert subspace of $c_\varphi$-invariant vectors and $C$ be the closed convex subset $(M_\varphi)_1$ of $\mathcal{H}_\varphi$.
In this situation  $\Pi$ denotes the nearest point projection from $\mathcal{H}_\varphi$ to the closed convex subset $(M_\varphi)_1$. As $c_\varphi$ acts by isometries, the subset $(M_\varphi)_1$ is $c_\varphi$-invariant and the projection $\Pi$ is $c_\varphi$-equivariant. Therefore $\Pi$ preserves the subspace $W$  of all $c_\varphi$-invariant vectors, as is required for the purpose of Lemma \ref{lem: a lemma on vector spaces}. 
\end{proof}

\section{Traces and spectral gap}\label{sec:Traces and spectral gap}

Let $\Gamma$ be a countable discrete group and $\varphi \in \Tr{\Gamma}$ be a fixed trace with associated GNS data $(\pi_\varphi,\rho_\varphi,\mathcal{H}_\varphi,v_\varphi)$. In addition,  consider the unitary conjugation  representation $c_\varphi : \Gamma \to \mathrm{U}(\mathcal{H}_\varphi)$ introduced in \S\ref{sec:conjugation representation}.
In this key section, we show that spectral gap for any of the representations $\pi_\varphi, \pi_\varphi \otimes \pi^*_\varphi$ and $c_\varphi$  can be expressed in terms of inequalities involving  elements of the group algebra $\mathbb{C}\left[\Gamma\right]$. The exact formulas will not be crucial. Rather, their  existence  demonstrates that spectral gap associated to traces is preserved when taking limits. This concept resonates with Kesten's theorem, though our analysis demands a more detailed approach.
%see   Proposition \ref{prop:summerizing-spectral-gap-limits} for a summary.

Fix an arbitrary element  $a \in \ell^1(\Gamma)$. Typically the element $a$ will be chosen so that its support generates the group $\Gamma$, but this is not necessary at this point. Set $b = a^* a \in \ell^1(\Gamma)$ and denote
\begin{equation}
b=\sum_{\gamma\in \Gamma} b_\gamma \gamma
\end{equation}
with the coefficients $b_\gamma \in \mathbb{C}$ for each $\gamma \in \Gamma$.
We will use these notations throughout this section without further mention.

In what follows we will repeatedly make use of the formula
\begin{equation}
\label{eq:norm in terms of trace}
\|\pi_\varphi(x) v_\varphi\|^2 = \left<\pi_\varphi(x) v_\varphi,\pi_\varphi(x) v_\varphi\right> = \left<\pi_\varphi(x^* x) v_\varphi, v_\varphi\right> = \varphi(x^* x)
\end{equation}
holding true for all elements $x \in \ell^1(\Gamma)$.

\begin{lemma}
\label{lem:bound on norm of pi}
Fix $\beta\geq 0$. Then
$\|\pi_\varphi(a)\| \le \beta^\frac{1}{2}$  if and only if
\begin{equation}
\label{eq:formula for spectral gap}
\varphi(x^* (b-\beta) x) \le 0
\end{equation} for every element $ x \in \mathbb{C}\left[\Gamma\right]$.
%then $\varphi$ does not dominate the trivial character.
\end{lemma}
\begin{proof}
Since the collection of vectors $\pi_\varphi(x) v_\varphi$ with $x \in \mathbb{C}\left[\Gamma\right]$ is dense in the Hilbert space  $\mathcal{H}_\varphi$, it suffices to estimate the norm of the operator $\pi_\varphi(a)$ on this subset. In particular, the estimate $\|\pi_\varphi(a)\| \le \beta^\frac{1}{2}$ is equivalent to the inequality
\begin{equation}
\label{eq:definition of spectral gap}
\|\pi_\varphi(a) \pi_\varphi(x) v_\varphi\| \le \beta^\frac{1}{2} \|\pi_\varphi(x) v_\varphi\|
\end{equation}
holding true for every $x \in \mathbb{C}\left[\Gamma\right]$. The formula in Equation (\ref{eq:norm in terms of trace}) shows that  Equation (\ref{eq:definition of spectral gap}) is equivalent for each individual element $x \in \mathbb{C}\left[\Gamma\right]$ to the inequality \begin{equation}
\label{eq:pre rearrange}
\varphi(x^* b x) = \varphi((ax)^*(ax)) \le \beta \varphi(x^* x).
\end{equation}
The desired formula in Equation (\ref{eq:formula for spectral gap}) follows by rearranging Equation (\ref{eq:pre rearrange}).
\end{proof}
The next estimate  is concerned with the conjugation representation $c_\varphi$.

\begin{lemma}
\label{lem:norm estimate for conjugation representation}
Fix   $\beta\geq 0$. Then
\begin{equation}\|c_\varphi(a)_{|\mathrm{span}_\mathbb{C}(v_\varphi)^\perp}\| \le \beta^{\frac{1}{2}}
\end{equation}
if and only if the inequality 
%then $\varphi \in \Ch{\Gamma}$.
\begin{equation}
\label{eq:formula for conjugation action}
 \varphi\left(\sum_{\gamma}b_{\gamma}\gamma^*x^*\gamma x-\beta x^* x    \right) \leq  (1-\beta) |\varphi(x)|^2
\end{equation}
holds true for every element $x \in \mathbb{C}\left[\Gamma\right]$
\end{lemma}
\begin{proof}
The collection of vectors of the form $\pi_\varphi(x) v_\varphi$ for some element $x \in \mathbb{C}\left[\Gamma\right]$ forms a dense subset of the Hilbert space $\mathcal{H}_\varphi$. Therefore, to estimate the norm of the operator $c_\varphi(a)$ it suffices to restrict attention to this subset. The operator norm bound $\|c_\varphi(a)_{|\mathrm{span}_\mathbb{C}(v_\varphi)^\perp} \|\le \beta^{\frac{1}{2}}$ is equivalent to the inequality
\begin{equation}
\label{eq:three terms}
\beta \|w\|^2     -
\|c_\varphi(a) w \|^2 \ge (\beta - 1) |\left<w, v_\varphi\right>|^2
\end{equation}
for all vectors of the form  $w = \pi_\varphi(x) v_\varphi$ for some $x \in \mathbb{C}\left[\Gamma\right]$, see the auxiliary    Lemma \ref{lem:condition for operator norm bound on complement of vector}. Let us estimate each one of the three terms appearing in Equation (\ref{eq:three terms}). First
\begin{equation}
\|w\|^2 = \|\pi_\varphi(x) v_\varphi\|^2 = \varphi(x^* x)
\end{equation} 
according to Equation (\ref{eq:norm in terms of trace}). Next we   compute $\|c_\varphi(a)w\|^2$. Observe that  
\begin{align} \label{eq: inner product for conjugation rep}
\begin{split}
\left\langle c_{\varphi}(\gamma)w,w\right\rangle &=\left\langle c_{\varphi}(\gamma)\pi_{\varphi}(x)v_{\varphi},\pi_{\varphi}(x)v_{\varphi}\right\rangle = \left\langle \pi_{\varphi}(\gamma)\rho_{\varphi}(\gamma)\pi_{\varphi}(x)v_{\varphi},\pi_{\varphi}(x)v_{\varphi}\right\rangle = \\
&= \left\langle \pi_{\varphi}(x^{*})\pi_{\varphi}(\gamma)\pi_{\varphi}(x)\rho_{\varphi}(\gamma)v_{\varphi},v_{\varphi}\right\rangle
=\varphi\left(\gamma^{*}x^{*}\gamma x\right)
\end{split}
\end{align}
for all elements   $\gamma\in \Gamma$. 
We have used the formula $\rho_\varphi(z) v_\varphi = \pi_\varphi(z^*) v_\varphi$  as well as the trace property $\tau(z z' ) = \tau(z' z )$ which hold true   for all elements $z,z'\in M_\varphi$. Therefore we get
\begin{equation}\label{eq: inner product for conjugation rep2}
\|c_{\varphi}(a)w\|^{2}=\left\langle c_{\varphi}(b)w,w\right\rangle =\sum_{\gamma\in \Gamma} \left\langle b_{\gamma} c_{\varphi}(\gamma)w,w\right\rangle =\varphi\left(\sum_{\gamma\in \Gamma}b_{\gamma}\gamma^*x^*\gamma x\right).
\end{equation}
Lastly the inner product $\left<w,v_\varphi\right>$ is simply given by
\begin{equation} 
\left<w,v_\varphi \right> = \left<\pi_\varphi(x) v_\varphi, v_\varphi\right> = \varphi(x).
\end{equation}
The desired formula in Equation (\ref{eq:formula for conjugation action}) follows by combining these three computations.
\end{proof}

The following two results are  extensions of Lemma \ref{lem:bound on norm of pi} to the tensor product representation $\pi_\varphi \otimes \pi^*_\varphi$. We identify $\pi_\varphi^*$ with the unitary representation associated to the conjugate trace $\overline{\varphi} \in \Tr{\Gamma}$  admitting the GNS data $(\pi_\varphi^*,\rho_\varphi^*,\mathcal{H}^*_\varphi,v^*_\varphi)$.

\begin{lemma} 
\label{lem:tensor-product-formula-infinite dimensional case}
Assume that the Hilbert space $\mathcal{H}_\varphi$ is infinite-dimensional. Fix \mbox{$\beta\geq 0$}.  Then
\begin{equation}
\label{eq:tensor norm estimate}
\|(\pi_\varphi  \otimes \pi_\varphi^*)(a)\| \le \beta^\frac{1}{2}
\end{equation}
if and only if
\begin{equation}
\label{eq:formula for tensor product}
 \sum_{i,j=1}^{k} \left( \sum_{\gamma\in\Gamma}b_{\gamma}\varphi(x_{j}^{*}\gamma x_{i})\overline{\varphi}\left(y_{j}^{*}\gamma y_{i}\right) -\beta  \varphi(x_j^* x_i) \overline{\varphi}(y_j^* y_i)\right) \le 0
\end{equation}
for every $k \in \mathbb{N}$ and every $x_1,\ldots,x_k \in \mathbb{C}\left[\Gamma\right]$ and $y_1,\ldots,y_k \in \mathbb{C}\left[\Gamma\right]$. 
%then $\varphi$ does not dominate a finite dimensional character.
\end{lemma}
%The quotient $\frac{1}{\dim \mathcal{H}_\varphi}$ should be interpreted as $0$ if $\mathcal{H}_\varphi$ is infinite-dimensional. In that case the Lemma holds with respect to the operator norm of $\|\pi_\varphi(a) \otimes \pi_\varphi^*(a)\|$.
\begin{proof}
The collection of all the vectors of the form
\begin{equation}
\label{eq:the vector w}
w = w^k_{x_i,y_i} = \sum_{i=1}^k \pi_\varphi(x_i) v_\varphi \otimes \pi^*_\varphi(y_i) v^*_\varphi
\end{equation}
for some $k \in \mathbb{N}$ and some $x_i,y_i \in \mathbb{C}\left[\Gamma\right]$ is a dense subset of the Hilbert space $\mathcal{H}_\varphi \otimes \mathcal{H}_\varphi^* \cong \mathrm{HS}(\mathcal{H}_\varphi)$. Therefore, the operator norm bound  $\|(\pi_\varphi  \otimes \pi_\varphi^*)(a)\|^2 \le \beta$ is equivalent to the inequality 
\begin{equation}
\label{eq:so desired}
    \|(\pi_\varphi  \otimes \pi_\varphi^*)(a)w\|^2\leq \beta\|w\|^2 
\end{equation}
holding true for every vector $w=w^k_{x_i,y_i}$ as in Equation (\ref{eq:the vector w}) above.  
It is possible to rewrite the norm of the vector $w$ in terms of the trace $\varphi$, namely
\begin{align}
\label{eq:norm of w}
\begin{split}
\|w\|^2 =\langle w,w\rangle
&= \sum_{i,j=1}^k   \left<\pi_\varphi(x_i) v_\varphi \otimes \pi^*_\varphi(y_i) v^*_\varphi, \pi_\varphi(x_j) v_\varphi \otimes \pi^*_\varphi(y_j) v^*_\varphi \right>  \\
&= \sum_{i,j=1}^k \left<\pi_\varphi(x_i) v_\varphi, \pi_\varphi(x_j) v_\varphi \right> \cdot \left<  \pi^*_\varphi(y_i) v^*_\varphi,  \pi^*_\varphi(y_j) v^*_\varphi \right>  \\
&= \sum_{i,j=1}^k \varphi(x_j^* x_i) \overline{\varphi}(y_j^* y_i).
\end{split}
\end{align}
The next step is to estimate the norm   $\|(\pi_\varphi  \otimes \pi_\varphi^*)(a)w\|^2 $. First,  note that for each element $\gamma\in \Gamma$ individually we have
\begin{align}
\label{eq:norm of pi tensor}
\begin{split}
\left\langle \pi_{\varphi}\otimes\pi_{\varphi}^*(\gamma)w,w\right\rangle  
 &=\sum_{i,j=1}^{k}\left\langle \pi_{\varphi}(\gamma)\pi_{\varphi} (x_{i})v_{\varphi},\pi_\varphi (x_j) v_{\varphi}\right\rangle \cdot \left\langle \pi_\varphi^* (\gamma) \pi_{\varphi}^*\left( y_{i}\right)v_{\varphi}^{*},\pi_\varphi^*(y_{j}) v_{\varphi}^{*}\right\rangle  \\
&=\sum_{i,j=1}^{k}\left\langle \pi_{\varphi}(x_{j}^{*}\gamma x_{i})v_{\varphi},v_{\varphi}\right\rangle \cdot \left\langle \pi_{\varphi}^*\left(y_{j}^{*}\gamma y_{i}\right)v_{\varphi}^{*},v_{\varphi}^{*}\right\rangle \\
&= \sum_{i,j=1}^{k}\varphi(x_{j}^{*}\gamma x_{i})\bar{\varphi}\left(y_{j}^{*}\gamma y_{i}\right).
\end{split}
\end{align}
Using the fact that $b=a^* a$ and summing Equation (\ref{eq:norm of pi tensor}) over all $\gamma$'s we get
\begin{align}
\begin{split}
\label{eq:needs some rearranging}
\|\pi_{\varphi}\otimes\pi_{\varphi}^*(a)w\|^{2}=\left\langle \pi_{\varphi}\otimes\pi_{\varphi}^*(b)w,w\right\rangle &=\sum_{\gamma\in\Gamma}b_{\gamma}\left\langle \pi_{\varphi}\otimes\pi_{\varphi}^*(\gamma)w,w\right\rangle\\ &=\sum_{i,j=1}^{k}\sum_{\gamma\in\Gamma}b_{\gamma}\varphi(x_{j}^{*}\gamma x_{i})\overline{\varphi}\left(y_{j}^{*}\gamma y_{i}\right).
\end{split}
\end{align}
The desired Equation (\ref{eq:so desired}) follows from the two Equations (\ref{eq:norm of pi tensor}) and (\ref{eq:needs some rearranging}) by rearranging the terms.
\end{proof}

We now consider the situation when unitary representation $\pi_\varphi$ is finite-dimensional. In this case, we only obtain a one-way implication.
%For every element $x \in \mathbb{C}\left[\Gamma\right]$ let $\mathrm{rk}(x)$ denote the number of non-zero terms in $x$.

\begin{lemma}
\label{lem:bound on norm of tensor - finite dimensional case}
Assume that the Hilbert space $\mathcal{H}_\varphi$ is 
finite-dimensional so that \mbox{$\mathrm{Id} \in \mathrm{HS}(\mathcal{H}_\varphi)$}. Fix $\beta\geq 0$.  If
\begin{equation}
\label{eq:norm bound f.d case}
\|(\pi_\varphi \otimes \pi_\varphi^*)(a)_{|\mathrm{span}_\mathbb{C}(\mathrm{Id})^\perp}\| \le \beta^\frac{1}{2}
\end{equation}
then
\begin{equation}
\label{eq:formula for tensor product - f.d}
\sum_{i,j=1}^{k}  \sum_{\gamma\in\Gamma}b_{\gamma}\varphi(x_{j}^{*}\gamma x_{i})\overline{\varphi}\left(y_{j}^{*}\gamma y_{i}\right) -\beta  \varphi(x_j^* x_i) \overline{\varphi}(y_j^* y_i)
\leq (1-\beta) \frac{\left(\sum_{i=1}^k \|x_i\|_1 \|y_i\|_1\right)^2 }{(\dim_\mathbb{C} \mathcal{H}_\varphi)^3}
\end{equation}
for every $k \in \mathbb{N}$ and every $x_1,\ldots,x_k \in \mathbb{C}\left[\Gamma\right]$ and $y_1,\ldots,y_k \in \mathbb{C}\left[\Gamma\right]$. 
\end{lemma}
\begin{proof}
%Write $\mathrm{HS}(\mathcal{H}) =  \mathcal{H}_0 \oplus \mathcal{H}_1$ where  $\mathcal{H}_0$ is the one-dimensional subspace spanned by the identity operator $\mathrm{Id}$.  
The operator norm bound in Equation (\ref{eq:norm bound f.d case}) is equivalent to  the inequality
\begin{equation}\label{eq:fd pi tensor estimate using aux lemma}
 \|(\pi_\varphi  \otimes \pi_\varphi^*)(a)w\|^2 -\beta \|w\|^2  \leq \frac{(1-\beta)}{\dim_\C \mathcal{H}_\varphi} |\left<w,\mathrm{Id}\right>|^2
\end{equation}
holding true for any vector $w \in \mathrm{HS}(\mathcal{H}_\varphi)$ of the form defined in Equation (\ref{eq:the vector w}), see the auxiliary  Lemma \ref{lem:condition for operator norm bound on complement of vector} below. We have obtained formulas for the norms $\|w\|^2$ as well as \mbox{$\|(\pi_\varphi \otimes \pi^*_\varphi)(a) w\|^2$} in terms of the trace $\varphi$, see Equations (\ref{eq:norm of w}) and (\ref{eq:norm of pi tensor}).

To conclude the proof it remains to estimate the inner product $\left<w,\mathrm{Id}\right>$. 
Recall that the inner product on the Hilbert--Schmidt space $\mathrm{HS}(\mathcal{H}_\varphi) $ is given explicitly   by \begin{equation}
    \langle T,S \rangle=\frac{1}{\dim_\C \mathcal{H}_\varphi} \tr (S^*T)
\end{equation} for each pair of operators $T,S \in \mathrm{HS}(\mathcal{H}_\varphi)$. In particular $\| \mathrm{Id} \|^2 = \dim_\C \mathcal{H}_\varphi$.

First consider the special case  $k=1, x = \delta_{\gamma_1}$ and $y = \delta_{\gamma_2}$ where $\gamma_1,\gamma_2 \in \Gamma$ are arbitrary elements so that $w = \pi_\varphi(\gamma_1) v_\varphi \otimes \pi^*_\varphi(\gamma_2) v^*_\varphi$ has rank $1$ as   an  operator on $\mathcal{H}_\varphi$. In that case
\begin{equation}
| \left<w,\mathrm{Id} \right> | =  |\left<\pi_\varphi(\gamma_1) v_\varphi \otimes \pi^*_\varphi(\gamma_2) v^*_\varphi, \mathrm{Id} \right>| \le \frac{1}{\dim_
\mathbb{C} \mathcal{H}_\varphi}.
\end{equation}
In the general case, relying on the triangle inequality we get
\begin{equation}
\label{eq:with triangle inequality}
| \left<w,\mathrm{Id} \right> | \le \frac{\sum_{i=1}^k \|x_i\|_1 \|y_i\|_1}{\dim_
\mathbb{C} \mathcal{H}_\varphi}.
\end{equation}
Substituting this Equation (\ref{eq:with triangle inequality}) in the previous Equation (\ref{eq:fd pi tensor estimate using aux lemma}), we get the desired inequality in Equation (\ref{eq:formula for tensor product - f.d}). 
\end{proof}

The following auxiliary computation was used twice in the above proofs.

\begin{lemma}
\label{lem:condition for operator norm bound on complement of vector}
Let $A$ be a bounded linear operator acting on the Hilbert space $\mathcal{H}$. Let $v \in \mathcal{H}$ be an $A$-invariant \emph{unit} vector. Fix some $0 < \beta < 1$. Then
\begin{equation}
\|A_{|\mathrm{span}_\mathbb{C}(v)^\perp}\| \le \beta^{\frac{1}{2}} 
\end{equation}
 if and only if
\begin{equation}
\label{eq:ugly}
\beta \|w\|^2 - \|Aw\|^2 \ge (\beta-1) |\left<w,v\right>|^2
\end{equation}
for all vectors $w \in \mathcal{H}$.
\end{lemma}
\begin{proof}
Denote $\mathcal{H}_0 = \mathrm{span}_\mathbb{C}\{v\}$. Consider the $A$-invariant direct sum decomposition $\mathcal{H} = \mathcal{H}_0 \oplus \mathcal{H}_1$ where $\mathcal{H}_1 = \mathcal{H}_0^\perp$. Any vector $w \in \mathcal{H}$ can be written uniquely as $w=w_0 + w_1$ with $w_0 = \left<w,v\right> v \in \mathcal{H}_0$ and $w_1 = w - w_0 \in \mathcal{H}_1$.

To begin with, assume that Equation (\ref{eq:ugly}) holds for  all vectors $w \in \mathcal{H}$. 
Note that every vector $w \in \mathcal{H}_1$ satisfies $\left<w,v\right> = 0$. In particular  Equation (\ref{eq:ugly}) reads \mbox{$\|Aw\|^2 \le \beta \|w\|^2$} with respect to any such  vector, as required.

Conversely, assume that $\|A_{|\mathcal{H}_1}\| \le \beta^\frac{1}{2}$. Therefore  any vector $w \in \mathcal{H}$ satisfies
\begin{equation}
\label{eq:pre ugly}
\|Aw\|^2 = \|A w_0\|^2 + \|A w_1\|^2 \le \|w_0\|^2 + \beta \| w_1\|^2 =  (1-\beta) \| w_0\|^2 + \beta \|w\|^2.
\end{equation}
 The desired Equation (\ref{eq:ugly}) follows from Equation (\ref{eq:pre ugly}) noting that $\|w_0\|=|\langle w,v\rangle|$.
\end{proof}

\section{Spectral gap and convergence of traces}
\label{sec:convergence of traces}

Let $\Gamma$ be a countable discrete group.  
We study the behaviour of various spectral gap properties with respect to taking limits of sequences of traces on the group $\Gamma$.

Let $\mathcal{R}$ be a set of equivalence classes of unitary representations of  $\Gamma$. For each unitary representation $\pi \in \mathcal{R}$  let $\pi_0$ denote its sub-representation  obtained by taking  the orthogonal complement to the subspace of invariant vectors.

\begin{defn*}
The group $\Gamma$ has \emph{property $\mathrm{T}(\mathcal{R})$} if the trivial representation $1 \in \widehat{\Gamma}$ is an isolated point of the set
\begin{equation}
    \{ \pi_0  \: : \: \pi \in \mathcal{R}  \} \cup \{1\}
\end{equation}
with respect to the Fell topology.
    
\end{defn*}
Property $\mathrm{T}(\mathcal{R})$ for the group $\Gamma$ is equivalent to   the following   condition:
\begin{itemize}
  %  \item If $1 \prec \oplus_{\pi_i\in \mathcal{R}} \pi_i$ then $1 \le \pi_i$ for some $i$.
    \item If $\rho$ is a unitary representation with $1 \le \rho$ then  $\rho \notin \overline{\{\pi_0 \: : \: \pi \in \mathcal{R} \}}$ with respect to the Fell topology.
\end{itemize}
In the special case where every representation $\pi \in \mathcal{R}$ is irreducible, i.e. the set of representations $\mathcal{R}$ is contained in the unitary dual $\widehat{\Gamma}$, the  above definition of property $\mathrm{T}(\mathcal{R})$ coincides with the one given by Lubotzky and Zimmer in \cite[Definition 1.1]{lubotzky1989variants}. For our purposes it will be more convenient to allow for reducible representations as well.

Here are some well known examples of property $\mathrm{T}(\mathcal{R})$. 
\begin{itemize}
    \item Let $\mathrm{Rep}_{\aleph_0}(\Gamma)$ be the set of equivalence classes of all separable unitary representations of the countable group $\Gamma$. Property $\mathrm{T}(\mathrm{Rep}_{\aleph_0}(\Gamma))$ is the same as Kazhdan's property (T).

%Recall that $\widehat{\Gamma}$ denotes the \emph{unitary dual} of the group $\Gamma$, i.e. the set of equivalence classes of all \emph{irreducible} unitary representations. The group $\Gamma$ has property $\mathrm{T}(\widehat{\Gamma})$ if and only if it has Kazhdan's property (T) \cite[Theorem 1.2.5]{bekka2008kazhdan}. 

\item 
Let $\mathrm{FD}$ be the family of all equivalence classes of finite-dimensional unitary representations of the group $\Gamma$. Property $\mathrm{T}(\mathrm{FD})$ has been considered in the literature \cite{lubotzky1989variants}.

\item 
Property $(\tau)$ can be regarded as  a special case of property $\mathrm{T}(\mathcal{R})$ for a particular choice of the set $\mathcal{R}$. Indeed, given a  family  $\mathcal{N}$ of normal subgroups of the group $\Gamma$ we denote by  $\mathcal{R}(\mathcal{N})$  the set of all unitary representation  factoring through some normal subgroup  in $ \mathcal{N}$. 
Property $(\tau)$ is the same as property $\mathrm{T}(\mathcal{R}(\mathcal{F}))$ where $\mathcal{F}$ is the family of all finite index normal subgroups of the group $\Gamma$. See \cite{lubotzky2005property} for more information.

\item
In the special case where $\Gamma$ is an arithmetic group, \emph{Selberg's property} is the same as property $\mathrm{T}(\mathcal{R}(\mathcal{C}))$ where  $\mathcal{C}$ is the family of all congruence subgroups \cite{clozel2003demonstration}. 
\end{itemize}

%In what follows we will require  the family of representations $\mathcal{R}$ to satisfy certain closure properties.

\begin{defn*}
The set $\mathcal{R}$ of equivalence classes of unitary representations is called \emph{hereditary} if 
\begin{enumerate}
 \item $\mathcal{R}$ is \emph{subrepresentation-closed}: if $\rho \in \mathcal{R}$ and $\pi \le \rho$ then $\pi \in \mathcal{R}$, %We aren't really using this assumption but better keep it because it is very mild
   \item $\mathcal{R}$ is \emph{dual-closed}: if  $\rho \in \mathcal{R}$ then the dual representation satisfies $\rho^*\in \mathcal{R}$,  
    \item $\mathcal{R}$ is \emph{diagonal-closed}: Let $\Pi$ be a  representation   of the product group $\Gamma \times \Gamma$ whose restriction   to each factor is contained in $\mathcal{R}$. Then  the   restriction of $\Pi$ to the diagonal embedding of $\Gamma$ is also contained in $\mathcal{R}$. 
\end{enumerate}
\end{defn*}

The set of finite-dimensional representations $\mathrm{FD}$ is hereditary. The set of unitary representations $\mathcal{R}(\mathcal{N})$ associated to the family $\mathcal{N}$ of normal subgroups is hereditary provided that $\mathcal{N}$ is closed under taking finite intersections.

\subsection*{Convergence of traces}

Let $\mathcal{R}$ be a fixed \emph{hereditary} set  of equivalence classes of unitary representations of the group  $\Gamma$. Denote
% \begin{equation}
% \TrFam{\Gamma}{\mathcal{R}} = \{ \varphi \in \Tr{\Gamma} \: : \: \pi_\varphi \in \mathcal{R} \}
% \end{equation}
% and
\begin{equation}
\ChFam{\Gamma}{\mathcal{R}} = \{ \varphi \in \Ch{\Gamma} \: : \: \pi_\varphi \in \mathcal{R} \}.
\end{equation}
 Note that the dual-closed and diagonal-closed conditions (which are part of the notion of a hereditary set) imply that any character $\varphi \in \ChFam{\Gamma}{\mathcal{R}}$ has in addition $\pi_\varphi \otimes \pi_\varphi^* \in \mathcal{R}$ as well as $c_\varphi \in \mathcal{R}$.

Fix an arbitrary element $a \in \ell^1(\Gamma)$ satisfying $\|a\|_1 = 1$ and such that the support of $a$ generates the group $\Gamma$. Note that the group $\Gamma$ has property $\mathrm{T}(\mathcal{R})$ if and only if there is some $0 < \beta < 1$ such that $\|\pi_0(a)\| \le \beta$ for all $\pi \in \mathcal{R}$.

\begin{prop}
\label{prop:Ch L is closed}
If the group $\Gamma$ has property $\mathrm{T}(\mathcal{R})$ then $\overline{\ChFam{\Gamma}{\mathcal{R}}} \subset \Ch{\Gamma}$.
\end{prop}
%If the group $\Gamma$ has property $\mathrm{T}(\mathcal{R})$ then any accumulation point of the set $\ChFam{\Gamma}{\mathcal{R}}$ is a character.
%\end{prop}
\begin{proof}
Consider the family of the conjugation representations  $c_\varphi$ corresponding to the characters $\varphi \in \ChFam{\Gamma}{\mathcal{R}}$. The subspace of $c_\varphi$-invariant vectors for  each such representation $c_\varphi$ is one-dimensional and is spanned by the cyclic vector $v_\varphi \in \mathcal{H}_\varphi$, see Corollary \ref{cor:c-invariant vectors are one-dimensional iff its a character}. The property $\mathrm{T}(\mathcal{R})$ assumption implies that there is a fixed constant $ 0 < \beta < 1$ (depending on the chosen element $a \in \ell^1(\Gamma)$) such that operator norm bound $\|c_\varphi(a)_{|\mathrm{span}_\C (v_\varphi)^\perp}\| \leq \beta$ holds true for all characters $\varphi \in \ChFam{\Gamma}{\mathcal{R}}$. The same operator norm bound holds true for any given limiting trace $\psi \in \overline{\ChFam{\Gamma}{\mathcal{R}}}$ by   Proposition \ref{lem:norm estimate for conjugation representation}. In particular, the subspace of invariant vectors in the unitary representation $c_\psi$ corresponding to the  trace $\psi$ is  spanned by the cyclic vector $v_\psi$. We conclude that $\psi \in \Ch{\Gamma}$   by the converse direction of Corollary \ref{cor:c-invariant vectors are one-dimensional iff its a character}.
\end{proof}

\begin{prop}
\label{prop:1 is not in the closure over L}
If the group $\Gamma$ has property $\mathrm{T}(\mathcal{R})$ then  the trivial character $1 \in \Ch{\Gamma}$ is an isolated point of the set $\ChFam{\Gamma}{\mathcal{R}} \cup \{1\}$.
\end{prop}
\begin{proof}
Consider the family of the unitary 
representations  $\pi_\varphi$ corresponding to the characters $\varphi \in \ChFam{\Gamma}{\mathcal{R}} \setminus \{1\}$. The representations $\pi_\varphi$ admit no invariant vectors, see Corollary \ref{cor:condition for a character to be trivial}.   The property $\mathrm{T}(\mathcal{R})$ assumption implies that there is a fixed constant $ 0 < \beta < 1$ (depending on the chosen element $a \in \ell^1(\Gamma)$) such that the operator norm bound $\|\pi_\varphi(a)\|\leq \beta$  holds true for  all characters $\varphi \in \ChFam{\Gamma}{\mathcal{R}} \setminus \{1\}$.  The same operator norm bound holds true for any given limiting trace $\psi \in \overline{\ChFam{\Gamma}{\mathcal{R}} \setminus \{1\}}$ by Lemma \ref{lem:bound on norm of pi}. In particular, the unitary representation $\pi_\psi$ corresponding to the  trace $\psi$ admits no invariant vectors. Moreover $\psi \in \Ch{\Gamma}$ by Proposition \ref{prop:Ch L is closed}. We conclude that $\psi \neq 1$ by the converse direction of Corollary \ref{cor:condition for a character to be trivial}.
\end{proof}

Recall the notion of \emph{amenable characters} defined in the last paragraph of \S\ref{sec:traces and characters}.

\begin{prop}
\label{prop:non-amenable-limit}
Assume that the group $\Gamma$ has property $\mathrm{T}(\mathcal{R})$. Let $\varphi_n \in \ChFam{\Gamma}{\mathcal{R}}$ be any sequence of characters converging to some $\varphi \in \Tr{\Gamma}$. Assume  that the characters $\varphi_n$ are either finite-dimensional and satisfy $ \dim_\mathbb{C} \mathcal{H}_{\varphi_n} \to \infty$ or are infinite-dimensional. Then the character $\varphi$ is non-amenable.
\end{prop}
\begin{proof}
First, note that  $\varphi \in \Ch{\Gamma}$ by Proposition \ref{prop:Ch L is closed}. Second, the representation $\pi_\varphi$ must be infinite dimensional by Proposition \ref{prop:finite dim is not an accumulation point restate} below. The property $\mathrm{T}(\mathcal{R})$ assumption implies that there is a fixed constant $0 < \beta < 1$ (depending on the chosen element $a \in \ell^1(\Gamma)$) such that
\begin{enumerate}
\item if the characters $\varphi_n$ are infinite-dimensional then the operator norm bound $\|\pi_{\varphi_n}\otimes \pi_{\varphi_n}^*(a)\| \leq\beta$  holds true for each $n \in \mathbb{N}$, and
\item if the characters $\varphi_n$ are finite-dimensional then the operator norm bound  \mbox{$\|\pi_{\varphi_n}\otimes \pi_{\varphi_n}^*(a)_{|\mathrm{span}_\mathbb{C}(\mathrm{Id}_n)^\perp}\| \leq\beta$} holds true for each $n \in \mathbb{N}$.
\end{enumerate}
We may apply Lemma \ref{lem:tensor-product-formula-infinite dimensional case} in the first case, or Lemma \ref{lem:tensor-product-formula-infinite dimensional case} in tandem with Lemma \ref{lem:bound on norm of tensor - finite dimensional case} in the second case, to deduce that the operator norm bound passes to the limiting character $\varphi$, in the sense that  $\|(\pi_\varphi \otimes \pi_\varphi^*)(a)\|\leq \beta  < 1$. Therefore  the limiting character $\varphi$ is non-amenable by definition.
\end{proof}

\subsection*{Faces and Bauer simplices} 
Assume that the countable group $\Gamma$ has property $\mathrm{T}(\mathcal{R})$. Then $\overline{\ChFam{\Gamma}{\mathcal{R}}} \subset \Ch{\Gamma}$ by Proposition \ref{prop:Ch L is closed}. To complete the proof of Theorem \ref{thm intro:Bauer-simplex-face-characters} of the introduction we will require the following abstract observation.

\begin{lemma}
\label{lem:Bauer-simplex-general}
Let $S$ be a metrizable Choquet simplex.
If  $X$ is a compact subset of the set of extreme points  $\partial S$ then $\overline{\conv X}$
 is a  face of the simplex $S$ and is a  Bauer simplex.
\end{lemma}

\begin{proof}
Denote $C = \overline{\conv X}$ so that $C$ is closed and convex. Note that 
\begin{equation}
\label{eq:containments KM}
X\subseteq C\cap \partial  S\subseteq\partial  C\subseteq\overline{X}.
\end{equation}
The rightmost inclusion in Equation (\ref{eq:containments KM}) follows from Milman's converse to the Krein--Milman theorem \cite{milman1947characteristics}.
Since $X$ is compact $X = \overline{X}$ and therefore equality must hold throughout.   In particular  $\partial C = X$.

We claim that the subset $C$ is a face of the simplex $S$. Suppose that $\mu$ is a probability
measure with $\mathrm{supp}(\mu) \subset \partial S$ such that its barycenter $\mathrm{bar}(\mu) = c $   satisfies $c \in C$.   By Choquet's
theorem applied to the convex set $C$ there exists a probability measure $\nu$ with $\mathrm{supp}(\nu)\subset \partial C = X $ whose barycenter is $\mathrm{bar}(\nu) = c$.  Since $S$ is a Choquet simplex it follows that $\mu = \nu$. In particular the measure $\mu$   is supported on $  \partial C$ so that $C$ is a face by definition.

To see that the convex set $C$ is a Choquet simplex note that as $\partial C \subset \partial S$ any two probability measures
supported on $\partial  C$ are in particular
supported on $\partial S$. Therefore if these two measures  have the same barycenter in $C$ they   must coincide. Finally the simplex $C$ is   Bauer since   $ \partial C = X $ is compact.
\end{proof}

\begin{proof}[Proof of Theorem \ref{thm intro:Bauer-simplex-face-characters}]
We have already established  that $\overline{\ChFam{\Gamma}{\mathcal{R}}} \subset \Ch{\Gamma}$ in Proposition \ref{prop:Ch L is closed}. The second part of the theorem saying that  $ \overline{\mathrm{conv} \, \Ch{\Gamma;\mathcal{R}}}$ is a closed face of the simplex $\Tr{\Gamma}$ and is a Bauer simplex follows abstractly by applying Lemma \ref{lem:Bauer-simplex-general} with respect to the subset $X = \overline{\ChFam{\Gamma}{\mathcal{R}}} $ of the Choquet simplex $S = \Tr{\Gamma}$.
\end{proof}

 %We also obtain Corollary \ref{cor intro: Kazhdan group has Bauer simplex} as a special case.

% \begin{proof}[Proof of Corollary \ref{cor intro: Kazhdan group has Bauer simplex}]
% Assume that the group $\Gamma$ has Kazhdan's property (T). This is equivalent to property $\mathrm{T}(\mathcal{R})$ with respect to  the set $\mathcal{R}$ of equivalence classes of all separable unitary representations of the group $\Gamma$. In particular $\ChFam{\Gamma}{\mathcal{R}} = \Ch{\Gamma}$. The desired conclusion follows at once from Proposition \ref{prop:Ch L is closed}.
% \end{proof} 
% \begin{thm}[
% \label{thm:Bauer-simplex-face-characters} in introduction]
%     Assume that $\Gamma$ has property $\mathrm{T}(\mathcal{R})$. Let $C\subseteq \Tr{\Gamma}$ be the closure of the convex hull of $\Ch{\Gamma;\mathcal{R}}$. Then $C$ is a closed face of  $\Tr{\Gamma}$, and it is a Bauer simplex.     
% \end{thm}

% \begin{proof}
% Apply Lemma  with $S=\Tr \Gamma$ and $X=\overline{\Ch{\Gamma;\mathcal{R}}}$.
% By Proposition \ref{prop:Ch L is closed} $X\subseteq\Ch \Gamma=\Ext S$ and so Lemma \ref{lem:Bauer-simplex-general}
% indeed applies. It is left to note that $C$ is also the closed convex hull of $X$. 
% \end{proof}

\subsection*{Finite-dimensional representations}

Recall that $\Gamma$ is a countable group and $\mathcal{R}$ is a hereditary set of equivalence classes of unitary representations of the group $\Gamma$.
We establish a finiteness result for  finite-dimensional characters in the presence of property $\mathrm{T}(\mathcal{R})$, see Corollary \ref{cor:tau implies finitely many reps in every dimension}.

\begin{lemma}
\label{lem:finite dim rep is isolated}
 Assume that the countable group $\Gamma$ has property $\mathrm{T}(\mathcal{R})$. Let $\rho  \in \overline{\mathcal{R}}$ be a finite-dimensional unitary representation of the group $\Gamma$. Then
 \begin{enumerate}
\item $ \rho \in \overline{ \{ \pi \in \mathcal{R} \: : \: \text{$\pi$ has a non-trivial finite-dimensional subrepresentation} \} }$.
\item If the representation $\rho$ is irreducible and $\rho \in \overline{\mathcal{R} \cap \widehat{\Gamma}}$  then  $\rho \in \mathcal{R}$ and $\rho$ is an isolated point of the subset $\overline{\mathcal{R} \cap \widehat{\Gamma}}  $ in the Fell topology.
\end{enumerate}
  \end{lemma}

\begin{proof}
Let $\rho \in \overline{\mathcal{R}}$ be a  finite-dimensional unitary representation. Consider a sequence $\pi_n \in \mathcal{R}$ of unitary representations such that $\pi_n \to \rho$ in the Fell topology. Therefore  $\rho \otimes \pi_n^* \to \rho \otimes \rho^*$ in the Fell topology.
As the unitary representation $\rho$ is finite-dimensional it satisfies    $1 \le \rho \otimes \rho^*$. In particular $\rho \otimes \pi_n^* \to 1$ in the Fell topology. It follows  that there is a subsequence  $m_n$ such that $\pi_{m_n} \otimes \pi^*_n \to 1$ in the Fell topology. Note that $\pi_{m_n} \otimes \pi_n^* \in \mathcal{R}$ for all $n \in \mathbb{N}$ as the family $\mathcal{R}$ is diagonal-closed. Property $\mathrm{T}(\mathcal{R})$ implies that   $ 1 \le \pi_{m_n} \otimes \pi^*_n$ for all $n$ sufficiently large. 
We conclude that the unitary representation $\pi_n$ admits a non-trivial finite-dimensional sub-representation for all $n$ sufficiently large \cite[Proposition A.1.12]{bekka2008kazhdan}. Item (1) follows.

Assume in addition that the unitary representations $\pi_n$ are all irreducible. This additional assumption implies by Schur's lemma that   $\pi_n$ is finite-dimensional and that $\pi_n = \pi_{m_n}$ for all $n \in \mathbb{N}$. However, we may  repeat the   argument of the previous paragraph with respect to the subsequence $m'_n = m_n + 1$. This  implies  by Schur's lemma that the sequence of  unitary representations $\pi_n$ stabilizes, namely there is some unitary representation $\pi \in \mathcal{R}$ with $\pi_n = \pi$ for all $n \in \mathbb{N}$ sufficiently large. The fact that the   sequence $\pi_n$ converges to the representation $\rho$ in the Fell topology is therefore equivalent to saying that $\rho \prec \pi$. 
Since the   representation $\pi$ is finite-dimensional and the representation $\rho$ is irreducible,   weak containment implies containment \cite[Corollary F.2.9]{bekka2008kazhdan}. Therefore  $\rho \le \pi$. As the representation $\pi$ is irreducible we conclude that $\rho = \pi \in \mathcal{R}$ and that  $\rho$ is an isolated point of the subset $\overline{\mathcal{R} \cap \widehat{\Gamma}}  $ in the Fell topology. Item (2) follows.
\end{proof}

Prior to proceeding we observe that the unitary representation given by the GNS construction depends continuously on the trace in question.

\begin{lemma}
\label{lem:taking character to its GNS is Fell continuous}
Let $\mathrm{Rep}_{\aleph_0}(\Gamma)$ denote the set of equivalence classes of separable unitary representations of the countable group $\Gamma$. The map 
\begin{equation}
  \Tr{\Gamma} \to \mathrm{Rep}_{\aleph_0}(\Gamma), \quad   \varphi \mapsto \pi_\varphi \quad \forall \varphi \in \Tr{\Gamma}
\end{equation}
is continuous with respect to the   pointwise convergence   topology on $\Tr{\Gamma}$ and the Fell topology on $ \mathrm{Rep}_{\aleph_0}(\Gamma)$. 
\end{lemma}
Essentially this lemma   is   a restatement in a different language of   \cite[Lemma F.1.3]{bekka2008kazhdan}. We explain this connection.

\begin{proof}[Proof of Lemma \ref{lem:taking character to its GNS is Fell continuous}]
Consider a sequence of traces  $\varphi_n \in \Tr{\Gamma}$ such that $\varphi_n\to \varphi$ pointwise for some trace $\varphi \in \Tr{\Gamma}$. We wish to show that the sequence of separable unitary representations $\pi_{\varphi_n}$ converges to the unitary representation $ \pi_{\varphi}$ in the Fell topology. This statement is equivalent to saying that $\pi_{\varphi} \prec \bigoplus \pi_{\varphi_n}$.  Certainly the \emph{particular} positive definite function $\varphi$ associated to the unitary representation $\pi_\varphi$ is a pointwise limit of the  positive definite functions $\varphi_n$ associated to the direct sum unitary representation $\bigoplus \pi_{\varphi_n}$.
However, in order to obtain the desired weak containment it suffices by  \cite[Lemma F.1.3]{bekka2008kazhdan} to consider \emph{only} the particular positive definite function associated to the cyclic vector $v_\varphi$.
\end{proof}

Certainly Lemma \ref{lem:taking character to its GNS is Fell continuous} applies equally well to the unitary representation $\rho_\varphi$ in place of $\pi_\varphi$.
The following is a restatement of Proposition \ref{prop:finite dim is not an accumulation point} from the introduction.

\begin{prop} 
\label{prop:finite dim is not an accumulation point restate}
Assume that the countable group $\Gamma$ has property $\mathrm{T}(\mathcal{R})$. Then any finite-dimensional character $\varphi \in \overline{\ChFam{\Gamma}{\mathcal{R}}}$ is an isolated point of the subset $\overline{\ChFam{\Gamma}{\mathcal{R}}}$. 
\end{prop}
\begin{proof}
Let $\varphi \in \overline{\ChFam{\Gamma}{\mathcal{R}}}$ be a finite-dimensional character. Consider a sequence of characters $\varphi_n \in  \ChFam{\Gamma}{\mathcal{R}}$ satisfying $\varphi_n \to \varphi$ in the pointwise convergence topology. The associated   unitary representations satisfy $\pi_{\varphi_n} \to \pi_\varphi$ in the Fell topology according to Lemma \ref{lem:taking character to its GNS is Fell continuous}. Moreover $\pi_{\varphi_n} \in \mathcal{R}$ by assumption so that $\pi_\varphi \in \overline{\mathcal{R}}$. Item (1) of Lemma \ref{lem:finite dim rep is isolated} implies that the representations $\pi_{\varphi_n}$ admit finite-dimensional sub-representations. According to Lemma \ref{lemma:condition for invariant vectors for tensor} and as the $\varphi_n$'s are characters we observe that the $\varphi_n$'s are finite-dimensional for all $n $ sufficiently large. 
 So $\pi_{\varphi_n} = d_n \rho_n$ and $\pi_{\varphi} = d \rho$ for some finite-dimensional irreducible unitary representations $\rho_n, \rho \in \widehat{\Gamma}$ where 
 \begin{equation}
d_n = \dim \rho_n, d = \dim \rho \quad\text{and} \quad  \varphi_n = \frac{1}{d_n} \mathrm{tr}\rho_n, \varphi = \frac{1}{d} \mathrm{tr} \rho.
 \end{equation}
 In this situation $\rho_n \to \rho$ in the Fell topology. It follows from Item (2) of Lemma \ref{lem:finite dim rep is isolated} that $\rho_n = \rho$ for all $n$ sufficiently large. So $\varphi_n = \varphi$ for all $n$ sufficiently large. We conclude that the character $\varphi$ is an isolated point as required.
\end{proof}

% In the special case where the group $\Gamma$ has Kazhdan's property (T) we derive Corollary \ref{cor intro: Kazhdan group has isolated fd chars}  from the above Proposition \ref{prop:finite dim is not an accumulation point restate} by taking  the family $\mathcal{R}$ to consist of all equivalence classes of separable unitary representations of the group $\Gamma$.  

\begin{cor}
\label{cor:tau implies finitely many reps in every dimension}
If the group $\Gamma$ has property $\mathrm{T}(\mathcal{R})$ then $\overline{\ChFam{\Gamma}{\mathcal{R}}}$ admits only finitely many characters in  each finite dimension.
\end{cor}
\begin{proof} 
%As the group $\Gamma$ is discrete its unitary dual $\widehat{\Gamma}$ is compact with respect to the Fell topology  \cite[Proposition 1.D.12]{bekka2020unitary}.
The subset $\widehat{\Gamma}_{\le d}$ of the unitary dual consisting of the equivalence classes of  irreducible  representations of dimension at most $d$ is closed  in the Fell topology for all $d \in \mathbb{N} $ \cite[Lemma 2.4]{fell1960dual}. 
% Put together this means that the subset $\overline{\mathcal{R}} \cap \widehat{\Gamma}_{\le d}$ is compact in the Fell topology for each $d \in \mathbb{N}$.
The Kolmogorov quotient\footnote{The \emph{Kolmogorov quotient} of a topological space is its quotient by the equivalence relation of topological indistinguishability. In other words the Kolmogorov quotient is the largest $T_0$ quotient.}   of the unitary dual  $\widehat{\Gamma}$ is the space $\mathrm{Prim}(\Gamma)$ consisting of weak equivalence classes of unitary representations \cite[1.E.3]{bekka2020unitary}. The quotient map to the  Kolmogorov quotient is closed \cite[Corollary 3.18]{pirttimaki2021survey}. Therefore the image of the subset $\widehat{\Gamma}_{\le d}$ in $\mathrm{Prim}(\Gamma)$ is closed for all $d \in \mathbb{N}$.  
Pulling   from $\mathrm{Prim}(\Gamma)$   back to $\Ch{\Gamma}$ (via the natural map  described in \cite[7.F.2]{bekka2020unitary}) and relying on Lemma \ref{lem:taking character to its GNS is Fell continuous}, we conclude that the subspace of $\overline{\ChFam{\Gamma}{\mathcal{R}}}$ consisting of all characters up to a given finite dimension is closed and hence compact. Each character in $\overline{\ChFam{\Gamma}{\mathcal{R}}}$ is an isolated point by Proposition \ref{prop:finite dim is not an accumulation point restate}. The finiteness of this subspace  follows.
\end{proof}

% %The map taking a character $\varphi \in \ChFam{\Gamma}{\mathcal{R}}$ to its GNS representation $\pi_\varphi \in \mathcal{R}$ is continuous by Lemma \ref{lemma:convergence of positive definite implies convergence of GNS}.

% Pulling back shows that the subset of $\ChFam{G}{\mathcal{R}}$ consisting of characters of some bounded dimension is compact. Each one is an isolated point. The conclusion follows.

\begin{rmk}
See  \cite[Proposition 2]{bass2002proalgebraic} as well as \cite[Proposition 4.1]{lubotzky2008representation} for a   related fact:  finite  abelianization implies finitely many representations factoring through a finite index kernel in each dimension.
\end{rmk}

\section{Character limits of groups with property $\mathrm{T}(\mathcal{R})$}
\label{sec:character-limits}

We study   groups enjoying various flavours of  property $\mathrm{T}(\mathcal{R})$ and their character theory.
The information gathered in the previous section \S\ref{sec:convergence of traces} leads to the following general statement on vanishing of characters.

\begin{thm}\label{thm:general-statement-for-vanishing-of-characters}
Let $\Gamma$ be a countable group and  $N \lhd \Gamma$ a normal subgroup. Let $\mathcal{R}$ be a hereditary set 
of equivalence classes of unitary representations of $\Gamma$. Assume 
    \begin{enumerate}
        \item \emph{(spectral gap)} the group $\Gamma$ has property $\mathrm{T}(\mathcal{R})$, and
        \item \emph{(character dichotomy)} any character of the group $\Gamma$ is either amenable or vanishes outside the subgroup $N$.
    \end{enumerate}
    Let  $\varphi_n\in \ChFam{\Gamma}{\mathcal{R}}$ by any sequence of pairwise distinct characters. Then any accumulation point  of the sequence $\varphi_n$ is a character supported on the subgroup $N$. Namely
    \begin{equation} \label{eq:vanishing-of-characters general}
        \quad \lim_{n\to \infty} \varphi_n(\gamma)=0 \quad \forall \gamma \in \Gamma \setminus N.
    \end{equation}
\end{thm}
\begin{proof}
%Consider a sequence $\varphi_n\in \Ch{\Gamma;\mathcal{R}}$   of pairwise distinct characters. 
Recall that the 
 space $\Tr{\Gamma}$ is compact. Let   $\varphi\in \Tr{\Gamma}$ be any accumulation point of the sequence $\varphi_n$. We know from  Proposition \ref{prop:Ch L is closed} that $\overline{\ChFam{\Gamma}{\mathcal{R}}} \subset \Ch{\Gamma}$. In particular the accumulation point   $\varphi$ is a \emph{character} of the group $\Gamma$.
 
The subset  $\ChFam{\Gamma}{\mathcal{R}}$ contains only finitely many distinct characters in each dimension, see Corollary \ref{cor:tau implies finitely many reps in every dimension}. We may assume up to passing to a subsequence either that the characters $\varphi_n$ are  all infinite-dimensional or     are all finite-dimensional and the dimension $\dim_{\mathbb{C}} \mathcal{H}_{\varphi_n}$   tends to infinity. In either case we conclude from Proposition \ref{prop:non-amenable-limit} that the accumulation point $\varphi$ is a \emph{non-amenable} character.
The  character dichotomy assumption   implies that $\varphi(\gamma) =0 $ for all elements $\gamma \in \Gamma \setminus N$.  Since  $\varphi$ was taken to be  an arbitrary accumulation point of the sequence  $\varphi_n$ in the topology of pointwise convergence  Equation (\ref{eq:vanishing-of-characters general}) holds true.
\end{proof}

The general principle formulated in Theorem \ref{thm:general-statement-for-vanishing-of-characters} can be  specialized to obtain the following useful result.

\begin{thm}
\label{thm: character rigid Kazhdan group}
Let $\Gamma$ be a group with property $\mathrm{T}(\mathrm{FD})$. Assume that   every character of $\Gamma$ is either finite-dimensional or vanishes outside the center   $\mathrm{Z}(\Gamma)$.  If $\varphi_n \in \Ch{\Gamma}$ is  any   sequence of pairwise distinct characters then $\varphi_n(\gamma) \rightarrow 0$ for all elements $\gamma \in \Gamma \setminus \mathrm{Z}(\Gamma)$.
\end{thm}

\begin{proof}
The asymptotic vanishing result in question follows immediately by applying   Theorem \ref{thm:general-statement-for-vanishing-of-characters} with respect to the family $\mathcal{R} = \mathrm{FD}$ and where the normal subgroup $N$ is taken to be the center $\mathrm{Z}(\Gamma)$ of the group $\Gamma$.
\end{proof}

Let us  explain how Theorem  \ref{thm: character rigid Kazhdan group} implies our main result, namely Theorem \ref{thm intro:higher rank lattice vanishing of characters} of the introduction. 

\begin{proof}[Proof of Theorem \ref{thm intro:higher rank lattice vanishing of characters}]
Let $G$ be a semisimple Lie group with finite center. Assume that $\mathrm{rank}_\mathbb{R}(G) \ge 2$ and that some non-compact simple factor of $G$ has property (T). Let $\Gamma$ be an irreducible lattice in the Lie group $G$. Every character of $\Gamma$ is either finite-dimensional of vanishes outside the center of $\Gamma$ by \cite{bader2021charmenability}. Furthermore, the lattice $\Gamma$ has property T(FD) by \cite{lubotzky1989variants}. This means that Theorem   \ref{thm: character rigid Kazhdan group} applies.
\end{proof}

In the next theorem we consider the situation in Example (2) on p.  \pageref{example:SL2}. 

\begin{thm}
\label{thm:example 2}
Let $k$ be an algebraic number field with ring of integers $\mathcal{O}_k$ and  $S \subset \mathcal{O}_k$ be a finite  (possibly empty) multiplicative  subset such that the localization $S^{-1} \mathcal{O}_k$ has infinitely many units. Consider the linear group  $\Gamma = \mathrm{SL}_2(S^{-1} \mathcal{O}_k)$.  If $\varphi_n \in \Ch{\Gamma}$ is any sequence of pairwise distinct characters then $\varphi_n(\gamma)\to 0$ for all elements $\gamma \in \Gamma \setminus \mathrm{Z}(\Gamma)$
\end{thm}
\begin{proof}
Every character of the  linear group $\Gamma$ is either finite-dimensional or vanishes outside the center   \cite[Theorem 2.6]{peterson2016character}. In fact, all finite-dimensional characters of the group $\Gamma$ come from  finite quotients, and  all finite quotients are  congruence quotients \cite{serre1970probleme}. Selberg's property \cite{clozel2003demonstration} implies that the group $\mathrm{SL}_2(S^{-1} \mathcal{O}_k)$ has property T(FD). As such Theorem   \ref{thm: character rigid Kazhdan group} applies.   
\end{proof}

Next, we deal with the case of congruence characters, i.e. characters factoring through a congruence subgroup. The following  justifies Example (3) on p. \pageref{example:congruence}.

\begin{thm}\label{thm:higher-rank-congruence-limit}
Let $G$ be a semisimple Lie group with $\mathrm{rank}_\mathbb{R}(G) \ge 2$ and finite center. Let  $\Gamma$ be an irreducible lattice in $G$ and  $\varphi_n \in \Ch{\Gamma}$ be any   sequence of pairwise distinct \emph{congruence} characters. Then $\varphi_n(\gamma)\to 0$ for all elements $\gamma\in \Gamma \setminus \mathrm{Z}(\Gamma)$. 
\end{thm}

\begin{proof}
Let $\Gamma$ be an irreducible   lattice in a higher rank semisimple Lie group. The lattice $\Gamma$ is arithmetic by the Margulis arithmeticity theorem.  The arithmetic group $\Gamma$ has Selberg's property \cite{clozel2003demonstration}. In our terminology this is the same as property   $\mathrm{T}(\mathcal{R}(\mathcal{C}))$ where $\mathcal{C}$ is the family of congruence subgroups of the lattice $\Gamma$. Any non-amenable character of the group $\Gamma$ vanishes outside of its center $\mathrm{Z}(\Gamma)$  by  \cite[Theorem A]{bbhp2022charmenability}. The desired vanishing result follows as a special case of the general Theorem \ref{thm:general-statement-for-vanishing-of-characters} with respect to the family $\mathcal{R} = \mathcal{R}(\mathcal{C})$ and where  the normal subgroup $N$ is taken to be the center $\mathrm{Z}(\Gamma)$ of the group $\Gamma$.
\end{proof}

\subsection*{Generalization to local fields}

Theorem \ref{thm intro:higher rank lattice vanishing of characters} can be generalized to $S$-arithmetic groups over local fields.  We now provide some additional details.

Let $K$ be a global field with ring of integers $\mathcal{O}$. Denote by $V^K$ denote the set of all equivalence classes of valuations on the field $K$ up to equivalence and by   $V^K_\infty \subset V^K$   the subset of Archimedean valuations.
Let   $S \subset V^K $ be a finite    set of valuations with $V^K_\infty \subset S$. The corresponding localization of the ring of integers $\mathcal{O}$ is
\begin{equation}
\label{eq:localization}
\mathcal{O}_S = \{ x \in K \: : \: |x|_s \le 1 \quad \forall s \in V^K \setminus S \}.
\end{equation}

Let $\bG$ be a connected  almost $K$-simple  algebraic $K$-group. Fix an injective $K$-representation $\rho : \bG \to \mathrm{GL}_N$ for some $N \in \mathbb{N}$. Denote 
\begin{equation}
\Gamma_S = \rho^{-1}(\mathrm{GL}_N(\mathcal{O}_S)).
\end{equation}
A subgroup $\Gamma$ of $\bG(K)$ is called \emph{$S$-arithmetic} if it is commensurable with $\Gamma_S$.  This notion is independent of the choice of the representation $\rho$. 
The $S$-arithmetic subgroup $\Gamma$ is said to be of 
%\emph{higher rank} if
%\begin{equation}
%\sum_{v \in V^K_\infty \cup S} \mathrm{rank}_{K_v} \bG %\ge 2
%\end{equation}
%where $V^K_\infty$ denotes the set of equivalence classes of all Archimedean valuations on the global field $K$.
 \begin{itemize}
     \item \emph{compact type} if $\rho(\Gamma_S)$ is compact in $\bG(K_v)$ for every absolute value $v$ on $K$,
     \item \emph{simple type} if there exists a unique absolute value $v$ on $K$ such that $\rho(\Gamma_S)$ is unbounded, and
     \item \emph{product type} otherwise.
 \end{itemize}
 The $S$-arithmetic subgroup $\Gamma$ is called \emph{higher rank} if $\Gamma$ has simple type and $\mathrm{rank}_{K_v} (\bG) \ge 2$ or if $\Gamma$ has product type.

\begin{thm} 
\label{thm:generalization to local fields}
Let $\Gamma$ be a higher rank $S$-arithmetic subgroup. Assume that the group $\bG(K_v)$ is non-compact and has property $(T)$  for some  absolute value $v$ on $K$.
Let  $\varphi_n\in \Ch{\Gamma}$ be any sequence of pairwise distinct characters. Then $\varphi_n(\gamma)\to 0$ for all $\gamma\in \Gamma\setminus \mathrm{Z}(\Gamma)$. 
\end{thm}
 
In particular, if the $S$-arithmetic group $\Gamma$  is center-free  then the conclusion of Theorem \ref{thm:generalization to local fields} says that  $\varphi_n\to \delta_e$ in the pointwise convergence topology.

\begin{proof}[Proof of Theorem \ref{thm:generalization to local fields}]
If $\Gamma$ has simple type then it has Kazhdan's property (T) \cite[Theorem 1.6.1]{bekka2008kazhdan}. If $\Gamma$ has product type then the assumption that $\bG(K_v)$ has property (T) for some absolute value $v$ on the global field $K$ guarantees that $\Gamma$ has   property $\mathrm{T}(\mathrm{FD})$ by \cite[Corollary 2.6.(i)]{lubotzky1989variants}. To be able to apply   \cite{lubotzky1989variants} we also need to use the fact that any non-compact group of the form $\bG(K_v)$  is \emph{minimally almost periodic}, i.e.   admits no non-trivial finite dimensional unitary representations  \cite{bader2017equicontinuous}.
Higher rank $S$-arithmetic groups have the character-dichotomy property by \cite[Theorem B]{bader2021charmenability}. 
At this point the desired conclusion  follows from our general  Theorem \ref{thm:general-statement-for-vanishing-of-characters}.
\end{proof}

This theorem has been mentioned in Example (1) on p. \ref{example:global fields} following Theorem \ref{thm intro:higher rank lattice vanishing of characters}.

\subsection*{Finite groups of Lie type}\label{sec:finite groups}

Let $\mathfrak{g}$ be a complex simple\footnote{This entire subsection makes  sense for semisimple rather than simple Lie algebras. We restrict attention to the simple case since it suffices for our purposes.} Lie algebra. The Lie algebra $\mathfrak{g}$ admits a   \emph{Chevalley basis} $B$. The basis $B$ is unique up to automorphisms of the   Lie algebra $\mathfrak{g}$ and sign changes \cite{steinberg1967lectures}. Let $\mathfrak{g}_
\mathbb{Z}$ be the $\mathbb{Z}$-span of the Chevalley basis $B$ so that $\mathfrak{g}_
\mathbb{Z}$ is a lattice in the complex vector space $\mathfrak{g}$.

Let $K$ be a field.  The \emph{Chevalley group} $G(K)$ associated to the Lie algebra $\mathfrak{g}$   is a connected $K$-simple adjoint\footnote{The general construction of   a Chevalley group depends on a choice of a lattice  containing the root lattice and contained in the lattice of weights. For our purposes it will suffice to consider the root lattice, which gives rise to adjoint Chevalley groups.} $K$-subgroup of the group $\mathrm{GL}(\mathfrak{g}_\mathbb{Z} \otimes K) \cong \mathrm{GL}_N(K)$ where $N = \dim_\mathbb{C}(\mathfrak{g})$.  The algebraic $K$-group $G(K)$ is $K$-split   and its   $K$-rank   is equal to the rank of the complex simple Lie algebra $\mathfrak{g}$.

Assume from now on that $K$ is a global field. Consider the localization  $\mathcal{O}_S$     of its ring of integers $\mathcal{O}$ corresponding to some finite subset $S$ of non-Archimedean valuations as in Equation \ref{eq:localization}. Regard the Chevalley group $G(K)$ equipped with the rational representation defined via the Chevalley basis $B$. In particular the subgroup $G(\mathcal{O}_S)$ consisting of the $\mathcal{O}_S$-rational points of the group $G(K)$ is the stabilizer of the submodule $\mathfrak{g}_\mathbb{Z} \otimes \mathcal{O}_S$. The group $G(\mathcal{O}_S)$ is $S$-arithmetic by definition.

%Let $\mathcal{O}_S$ be a ring of $S$-integers in $K$. Regard the Chevalley basis $B$ to define the representation $\rho : G \to \mathrm{GL}_N(K)$ where $N = |B| = \dim_\mathbb{C} \mathfrak{g}$. Then $G(\mathcal{O}_S)$ coincides with the stabilizer of $\mathfrak{g}_\mathbb{Z} \otimes \mathcal{O}_S$. It is an arithmetic group (by definition).

Consider any non-zero ideal $\mathcal{I} \le \mathcal{O}_S$ such that $F = \mathcal{O}_S/\mathcal{I}$ is a finite ring (or  field). Let $G(F)$ be the Chevalley group over the ring $F$. Since the ring $F$ is finite the Chevalley group $G(F)$ is known to coincide   with its elementary subgroup   generated by the root subgroups \cite[Lemma III.1.4]{weibel2013k}. This fact implies  that the natural homomorphism $f_\mathcal{I} : G(\mathcal{O}_S) \to G(F)$ is surjective \cite[Remark III.1.2.3]{weibel2013k}. 

We utilize the discussion in the preceding paragraphs to obtain character limit results for finite simple groups of Lie type.

\begin{proof}[Proof of Corollary \ref{cor intro: character limits}]
Let $G$ be a Chevalley group associated to some complex simple Lie algebra of rank two or higher.
There are two families of fields to consider.

In the case where the global field $K$ is a finite extension of $\Q$ we consider the local field $k=K_\infty$ which is either $\R$ or $\C$. 

In the case where $K$ is a finite extension of the field of rational function over a finite field $\mathbb F _q$, we may  identify (up to an isomorphism) the ring $\mathcal{O}_K$  with the localization $\mathcal{O}_S = \mathbb{F}_q\left[t^{-1}\right]$ of the ring of algebraic integers $\mathcal{O}_K$ associated to the subset of valuations  $S = \{ t \}$. We consider the local field  $k = K_t = \mathbb{F}_p\left(\left(t\right)\right)  $.

In either case the $k$-analytic group 
   $G(k)$   has $k$-rank at least two. In particular the group $G(k)$ is non-compact and has property $(T)$. In particular,  Theorem  \ref{thm:generalization to local fields} applies to the lattice $G(\mathcal{O}_K)$.

To conclude the proof, consider an arbitrary enumeration $\varphi_n$ of all characters  of the group $G(\mathcal{O}_K)$ which factorize through some finite quotient of the form $G(\mathcal{O}_K/I)$ where  $I$ is an ideal in $\mathcal{O}_K$. According to Theorem  \ref{thm:generalization to local fields}, for  any $\varepsilon>0$  there is some sufficiently large $N>0$ such that $|\varphi_n(\gamma)|<\varepsilon$ for any non-central element $\gamma\in G(\mathcal{O}_K)$. The desired statement follows by excluding finitely many ideals $I \lhd \mathcal{O}_K$.
\end{proof}

\subsection*{Generalization to non-semisimple groups}
Any discrete group $\Gamma$ admits a maximal amenable normal subgroup $\mathrm{Rad}(\Gamma)$ called the \emph{amenable radical}.
The following statement generalizes Theorem \ref{thm intro:higher rank lattice vanishing of characters} to the non-semisimple setting; see Example (4) on page \pageref{example:non semisimple}.

\begin{thm}\label{thm:general arithmetic}
Let $\bG$ be a connected $\Q$-algebraic group with solvable radical $\bR$. Assume that the semisimple group $\bG/\bR$ has at most one $\Q$-simple $\R$-anistorpic factor and no $\R$-simple factors of rank one. Let $ \Gamma \leq \bG(\Q)$ be an arithmetic subgroup with property (T). Let $\varphi_n \in \Ch{\Gamma}$ be any sequence of pairwise distinct characters. Then $\varphi_n(g)\to 0$ for any element $\gamma\in \Gamma \backslash \mathrm{Rad}(\Gamma)$.
\end{thm}

\begin{proof}
The crucial ingredient here is  \cite[Theorem A]{bader-iti2022charmenability}. Taking into account  the property (T) assumption, it says that any character of the arithmetic group $\Gamma$ is either finite-dimensional or vanishes outside   the amenable radical   $\Rad{\Gamma}$. We conclude  by relying on our   general Theorem \ref{thm:general-statement-for-vanishing-of-characters}, taking $\mathcal{R}$ to be the set of equivalence classes of all separable unitary representations and letting $N=\Rad{\Gamma}$.
\end{proof}

\subsection*{Convergence of representations}

We prove the  results stated in the introduction concerning convergence of representations in the Fell topology. 

\begin{defn*}
    Let $\pi_1$ and $\pi_2$ be two unitary representations of a group $\Gamma$.
    \begin{itemize}
        \item $\pi_1$ is \emph{subordinate} to $\pi_2$ if every non-trivial subrepresentation of $\pi_1 $ contains a non-trivial subrepresentation equivalent to a subrepresentation of $\pi_2$.
        \item $\pi_1$ and $\pi_2$ are \emph{quasi-equivalent} if each one is subordinate to the other.
        \item Quasi-equivalence is denoted $\pi_1 \approx \pi_2$.
    \end{itemize} 
\end{defn*}

Note that a finite-dimensional unitary representation $\pi_1$ is subordinate to a unitary representation $\pi_2$ if every irreducible subrepresentation of $\pi_1$ is equivalent to  a subrepresentation of $\pi_2$.

\begin{proof}[Proof of Theorem \ref{thm intro: convergence of reps}]
Let $\Gamma$ be a center-free irreducible lattice as in Theorem \ref{thm intro: convergence of reps}. As was already mentioned above, the lattice 
$\Gamma$ has property $\mathrm{T}(\mathrm{FD})$ by \cite{lubotzky1989variants} and  every character of $\Gamma$ other than $\delta_{\{e\}}$ is   finite-dimensional by \cite{bader2021charmenability}.

Consider a sequence $\pi_n$ of  finite-dimensional unitary representations of the group $\Gamma$.
Assume that the representations $\pi_n$ are pairwise non quasi-equivalent. Note that for every finite-dimensional unitary representation $\pi$, there are only finitely many quasi-equivalence classes of finite-dimensional unitary representations subordinate to $\pi$.
Therefore, up to passing to a subsequence, we may find for each $n \in \mathbb{N}$  an \emph{irreducible} subrepresentation $\sigma_n \le  \pi_n$ such  that the representations $\sigma_n$ are pairwise non-equivalent.
It follows that the characters $\varphi_n = \frac{1}{\dim \sigma_n} \mathrm{tr} \circ \sigma_n$  are pairwise distinct.  At this point we invoke Theorem \ref{thm: character rigid Kazhdan group}.
It follows that the sequence of characters  $\varphi_n$ converges to the Dirac trace $\delta_e$ pointwise. Note that the GNS construction associated to the trace  $\delta_e$ is the regular representation $\lambda_\Gamma$. 
Recall that the GNS construction depends continuously on the trace, see 
 Lemma \ref{lem:taking character to its GNS is Fell continuous}. Hence the sequence of unitary representations $\sigma_n$, and a fortiori also the sequence of representations $\pi_n$, converges to the regular representation $\lambda_\Gamma$ in the Fell topology. 
\end{proof}

\begin{proof}[Proof of Corollary \ref{cor:for descending stuff}] 
Let $\Gamma$ be a center-free  lattice as in Theorem \ref{thm intro: convergence of reps}. We will prove that the sequence of  quasi-regular representations $L^2(\Gamma/\Gamma_n)$ converges to the regular representation $\lambda_\Gamma$ in the Fell topology by showing that every subsequence contains a further subsequence with  limit $\lambda_\Gamma$. 
Let $n_k$ be an arbitrary subsequence. We claim that the sequence of unitary representations $L^2(\Gamma/\Gamma_{n_k})$ contains a further subsequence consisting of pairwise non quasi-equivalent representations. If this is not the case then $\bigoplus_k L^2(\Gamma/\Gamma_{n_k})$ is quasi-equivalent to some finite-dimensional representation $\rho$ with $  \ker \rho = N$ satisfying $\left[\Gamma:N \right] < \infty$. Quasi-equivalent representations have the same kernel \cite[Proposition 6.A.4]{bekka2020unitary}. In particular  $N \le \Gamma_{n_k}$ for all $k \in \mathbb{N}$. This  leads to a contradiction to the fact the sequence $\Gamma_{n_k}$ is strictly decreasing. 
Having established the claim, the convergence of  this further  subsequence  to the   representation $\lambda_\Gamma$  follows from Theorem \ref{thm intro: convergence of reps}.
\end{proof}

% \begin{proof}[Proof of Corollary] 
% %\ref{cor intr: weak convergence property}]
%     Let $G$ be a centre-free semisimple Lie group admitting a non-compact simple factor with property (T) and having $\mathrm{rank}_\mathbb{R}(G) \ge 2$. Let $\Gamma$ be an irreducible lattice in the Lie group $G$. The lattice $\Gamma$ is 
%     has the character dichotomy property by\cite{bader2021charmenability} and has property T(FD) by \cite{lubotzky1989variants}.
%         Let $\Gamma_n \le \Gamma$ be a strictly descending sequence of subgroups. The sequence of quasi-regular representations $L^2(\Gamma/\Gamma_n)$ converges in the Fell topology to the regular representation $\lambda_\Gamma$ by Corollary \ref{cor:for descending stuff}. The desired conclusion follows by inducing representations from the lattice $\Gamma$ to the Lie group $G$ and using the continuity of induction (see e.g. \cite[Theorem F.3.5]{bekka2008kazhdan}).
% \end{proof}

We also identify all possible limits of sequences of finite-dimensional unitary representations of lattices.

  \begin{prop}
\label{prop:limits of finite dim}
 Let $\Gamma$ be a center-free irreducible lattice as in Theorem \ref{thm intro: convergence of reps}. Let $\pi_n$ be a sequence of pairwise disjoint finite-dimensional unitary representations of  $\Gamma$. Then any limit  of the sequence $\pi_n$ in the Fell topology weakly contains the regular representation $\lambda_\Gamma$.
%At the same time, by Theorem \ref{thm intro: convergence of reps}, the regular representation is a limit point of $\pi_n$.      
\end{prop}

\begin{proof}
Let $\pi$ be any limit point of the sequence of representations $\pi_n$ in the Fell topology. We claim that the representation $\pi$ is \emph{weakly mixing}, i.e. does not contain any finite-dimensional subrepresentations. 
 Assume by contradiction that $\pi$ does contain some  irreducible finite-dimensional subrepresentation $\rho$. Up to passing to a subsequence, we may assume that $\rho \not\le \pi_n$ for all $n \in \mathbb{N}$. Let $\sigma_k$ be some enumeration  of all finite-dimensional irreducible subrepresentations contained in some representation $\pi_n$. Then 
 \begin{equation}
 \rho<\pi \prec \bigoplus_n \pi_n \approx \bigoplus_k \sigma_k
 \end{equation}
In particular, the irreducible finite-dimensional representations $\sigma_k$   converge to the representation $\rho$ in the Fell topology.  The lattice $\Gamma$ has property T(FD) by \cite{lubotzky1989variants}. This is a contradiction to Lemma \ref{lem:finite dim rep is isolated}. As the representation $\pi$ is weakly mixing, it follows from \cite[Theorem B]{bader2021charmenability} and from \cite[Proposition 3.6]{bbhp2022charmenability} that $\pi$ weakly contains the regular representation $\lambda_\Gamma$. 
\end{proof}

% \begin{proof}[Proof of Corollary \ref{cor:weak containment implies equivalent}]
%     Let $\Gamma$ be a  lattice as in the previous Proposition \ref{prop:limits of finite dim}. Take an arbitrary strictly decreasing sequence $\Gamma_n$ of finite-index subgroups of $\Gamma$. We know from Corollary \ref{cor:for descending stuff} that the unitary quasi-regular representations $L^2(\Gamma/\Gamma_n)$ converge to the regular representation $\lambda_\Gamma $ in the Fell topology.
%         Consider any unitary representation $\pi$ of $\Gamma$   weakly contained in the regular representation $\lambda_\Gamma$. It follows from the   definition of the Fell topology that $\pi$ is also a limit of the sequence $L^2(\Gamma/\Gamma_n)$. Therefore the representation $\pi$ weakly contains the regular representation $\lambda_\Gamma$ by Proposition \ref{prop:limits of finite dim}. We conclude that the two representations $\pi$ and $\lambda_\Gamma$ are weakly equivalent, as required.
% \end{proof}

We remark that the above theorems concerning the convergence of representations naturally extend to semisimple algebraic groups defined over arbitrary local fields and their $S$-arithmetic subgroups; see the discussion in  Theorem \ref{thm:generalization to local fields}.
%dealing with   naturally extend to the current discussion of convergence of representations in the Fell topology.

\section{Relative traces and characters}\label{sec:rel-traces}

Let $\Gamma$ be a countable discrete group. Let  $\Lambda$ be another countable discrete group admitting an action by automorphisms $\theta:\Lambda\to \Aut{\Gamma}$  on the group $\Gamma$.  This induces an action of the group $\Lambda$ on the set of traces $\Tr{\Gamma}$ by precomposition.   The \emph{relative traces} of the group $\Gamma$ are given by
\begin{equation}
\relTr{\Lambda}{\Gamma} = \{ \varphi \in \Tr{\Gamma} \: : \: \varphi^\lambda = \varphi \quad \forall \lambda \in \Lambda \}.   
\end{equation}
The set of relative traces $\relTr{\Lambda}{\Gamma}$ is a simplex \cite[Proposition 2.2]{bader-iti2022charmenability}.
Its   extreme points are called \emph{relative characters} and are denoted by $\relCh{\Lambda}{\Gamma}$. 

 The  set $\relTr{\Lambda}{\Gamma}$ stands in a bijective correspondence with the set of all $\Lambda$-invariant Borel probability measures on $\Ch{\Gamma}$ via the barycenter map. The relative characters $\relCh{\Lambda}{\Gamma}$ correspond to ergodic  $\Lambda$-invariant measures on the space $\Ch{\Gamma}$.

Consider the situation where the group $\Lambda$ has Kazhdan's property (T) and $\Gamma$ is a normal subgroup of $\Lambda$. The group $\Lambda$ acts on its subgroup $\Gamma$ by conjugation. If  $\Tr{\Gamma}$ happens to be a Bauer simplex so that the space $\Ch{\Gamma}$ is compact then it is an immediate consequence of the above discussion, together  with the   theorem of Glasner and Weiss \cite{glasner1997kazhdan}, that the space of relative traces $\relTr{\Lambda}{\Gamma}$ is also a Bauer simplex.  For example, this will be the case if the group $\Gamma$ is abelian or has Kazhdan's property (T) in its own right.

The main point of the remainder of this section will be to obtain a similar conclusion assuming that the group $\Lambda$ has property (T) but without  assuming that the simplex $\Tr{\Gamma}$ is Bauer.

% For instance, in the situation where 

% \begin{cor}
% \label{cor: \Gamma and \Gamma/N have T}
% Let $N$ be a normal subgroup of $\Gamma$, and assume that the quotient $\Gamma/N$ has Kazhdan property $(T)$. Moreover assume the simplex of traces $\Tr{N}$ is Bauer, e.g., either $N$ is abelian or has Kazhdan property $(T)$. Then 
% \end{cor}

\subsection*{Relative traces and the GNS data}
Consider   a relative trace $\varphi \in \relTr{\Lambda}{\Gamma}$ with corresponding GNS data ($\pi_\varphi,\rho_\varphi,\mathcal{H}_\varphi,v_\varphi)$. There is a unique unitary representation $U_\varphi : \Lambda \to \mathrm{U}(\mathcal{H}_\varphi)$ given by $U_\varphi : \lambda \mapsto U_\varphi (\lambda)$ which satisfies
\begin{equation}
    U_\varphi (\lambda) \pi_\varphi (\gamma) U_\varphi (\lambda)^{-1} = \pi_\varphi(\gamma^\lambda) \quad \text{and} \quad U_\varphi (\lambda) v_\varphi = v_\varphi \quad \forall \lambda \in \Lambda, \gamma \in \Gamma.
\end{equation}
 Note that $U_\varphi (\lambda) = \pi_\varphi(\gamma) $ provided that $\theta(\lambda)=\mathrm{Inn}(\gamma)$, namely the element $\lambda \in \Lambda$ acts on the group $\Gamma$ via the inner automorphism corresponding to the element $\gamma \in \Gamma$. For all this see \cite[Proposition 2.7]{bekka2020characters}.
 
Let $N_\varphi$ denote the von Neumann algebra generated by the operators $\pi_\varphi(\Gamma)$ as well as   the $U_\varphi (\lambda)$'s for all elements $\lambda \in \Lambda$, namely
\begin{equation}
N_\varphi = \left( \pi_\varphi(\Gamma) \; \cup \; U_\varphi (\Lambda) \right)''.
\end{equation}
Then $N_\varphi' = \mathbb{C} \cdot \mathrm{Id}$ if and only if $\varphi \in \relCh{\Lambda}{\Gamma}$ \cite[Proposition 2.8]{bekka2020characters}.

\begin{prop}
\label{prop:understanding the center} Assume that $\mathrm{Inn}(\Gamma) \le \theta(\Lambda)$. Then   any  relative trace $\varphi \in \relTr{\Lambda}{\Gamma}$ satisfies
\begin{equation}
N'_\varphi =  \mathrm{Z}(M_\varphi) \cap U_\varphi (\Lambda)'.
\end{equation}
\end{prop}
\begin{proof}
The inclusion in the following direction is immediate
\begin{equation}
\mathrm{Z}(M_\varphi) \cap U_\varphi (\Lambda)' \leq \pi_\varphi(\Gamma)' \cap U_\varphi (\Lambda)'  =
N'_\varphi.
\end{equation}
Let us argue in the converse direction and show that $N'_\varphi \le \mathrm{Z}(M_\varphi) \cap U_\varphi (\Lambda)'$. Consider any element $T \in N'_\varphi$.
Then $\varphi_{T^\frac{1}{2}}$ is a $\Lambda$-invariant positive-definite function on the group $\Gamma$ dominated by the trace $\varphi$   \cite[Proposition 2.8]{bekka2020characters}. Since $\theta(\Lambda)$ contains all inner automorphisms,  $\varphi_{T^\frac{1}{2}} \in \relTr{\Lambda}{\Gamma}$. The element $T$ with these properties is uniquely determined in $N'_\varphi$. On the other hand $\varphi_{T^\frac{1}{2}} = \varphi_{S^\frac{1}{2}}$ for some element $S \in \mathrm{Z}(M_\varphi)$ with $0 \le S \le 1$. The fact that $\varphi_{T^\frac{1}{2}} = \varphi_{S^\frac{1}{2}}$ is $\Lambda$-invariant implies that $S \in U_\varphi (\Lambda)'$, see the first paragraph on \cite[p.11]{bekka2020characters}. We obtain $S \in N'_\varphi$. The uniqueness of the element $T$ as above means that $T = S$ so that $T \in \mathrm{Z}(M_\varphi)$. As $T \in U_\varphi (\Lambda)'$ the desired result follows.
\end{proof}

% \begin{lemma}
% Let $x \in M_\varphi$ and $\lambda \in \Lambda$. Then $y = \alpha_\varphi(\lambda) x \in M_\varphi$.
% \end{lemma}
% \begin{proof}
% We will show that $y \in \pi_\varphi(\Gamma)''$. Let $z \in \pi_\varphi(\Gamma)'$. Then
% \begin{equation}
% yz = (\alpha_\varphi(\lambda)x) z = U_\varphi (\lambda) x U_{\lambda}^{-1} z = 
% \end{equation}
% \end{proof}

As was explained in \S\ref{sec:conjugation representation}, the mapping 
\begin{equation}
\iota_\varphi : M_\varphi \to \mathcal{H}_\varphi, \quad \iota_\varphi : x\mapsto x v_\varphi \quad \forall x \in M_\varphi
\end{equation}
allows us to regard the von Neumann algebra $M_\varphi$ as a dense subspace of the Hilbert space $\mathcal{H}_\varphi$.
For every element $\lambda \in \Lambda$ consider the map
\begin{equation}
\alpha_\varphi(\lambda) : M_\varphi \to M_\varphi, \quad \alpha_\varphi(\lambda) x = U_\varphi (\lambda) x U_\varphi (\lambda)^{-1}.
\end{equation}
For each element $\lambda \in \Lambda$ the operator $\alpha_\varphi(\lambda)$ preserves the norm of vectors in the image of the map $\iota_\varphi$, in the sense that
\begin{equation}
\| \iota_\varphi(x) \|_{\mathcal{H}_\varphi} = \| \iota_\varphi(\alpha_\varphi(\lambda) x) \|_{\mathcal{H}_\varphi} \quad \forall x \in M_\varphi.
\end{equation}
As such the operator $\alpha_\varphi(\lambda)$ extends to a unitary operator on the Hilbert space $\mathcal{H}_\varphi$. This process determines a unitary representation $\alpha_\varphi : \Lambda \to \mathrm{U}(\mathcal{H}_\varphi)$. Observe that an element $x \in M_\varphi$ satisfies $x \in U_\varphi (\Lambda)'$ if and only if  $\iota_\varphi(x) = x v_\varphi \in \mathcal{H}_\varphi$ is an invariant vector for the unitary representation $\alpha_\varphi$.

\begin{prop}\label{prop: relative character by dimension of invariant space}
Assume that $\mathrm{Inn}(\Gamma) \le \Lambda$. Let $\varphi \in \relTr{\Lambda}{\Gamma}$ be a relative trace. Then $\varphi \in \relCh{\Lambda}{\Gamma}$ if and only if $\dim_\mathbb{C} \mathcal{H}_\varphi^{\alpha_\varphi} = 1$.
\end{prop}
\begin{proof}
The same proof as that of Proposition \ref{prop:embedding with dense image of H^c into Z(M)} relying on Lemma \ref{lem: a lemma on vector spaces} shows that the commutant $N_\varphi'$ is dense in the 
Hilbert subspace $\mathcal{H}_\varphi^{\alpha_\varphi}$. 
Note that any element $\lambda \in \Lambda$ which acts on the group $\Gamma$ via an inner automorphism of the form $\mathrm{Inn}(\gamma)$ for some element $\gamma \in \Gamma$ satisfies $\alpha_{\varphi}(\lambda) = c_\varphi(\gamma)$ where $c_\varphi : \Gamma \to \mathrm{U}(\mathcal{H}_\varphi) $ is the conjugation representation studied in \S\ref{sec:conjugation representation}. The desired conclusion follows by using the aforementioned fact  that the   relative trace $\varphi $ is a relative character if and only if $N'_\varphi = \C \cdot \mathrm{Id}$.
\end{proof}

Assume that  $\lambda\in \Lambda$ is an element satisfying $\theta(\lambda)=\mathrm{Inn}(g)$ for some element $\gamma \in \Gamma$. In this case we have already seen  that $U_\varphi(\lambda)=\pi_\varphi(\gamma)$. As a result $\alpha_\varphi(\lambda)=c_\varphi(\gamma)$ holds true. 
Therefore Proposition \ref{prop: relative character by dimension of invariant space} can be seen as a generalization to the relative setting of the statements in \S\ref{sec:conjugation representation} dealing with the conjugation representation.

\subsection*{Spectral gap for the representation $\alpha_\varphi$}
The action by automorphisms of the group $\Lambda$ on the group $\Gamma$  induces an action on its group ring $\C[\Gamma]$. We denote this action by $x \mapsto x^\lambda$ for any element $\lambda \in \Lambda$ and $x \in \mathbb{C}\left[\Gamma\right]$.

Fix an arbitrary element  $a \in \ell^1(\Lambda)$  with $\|a\|_1 = 1$ whose support generates the group $\Lambda$.
Denote $b = a^* a$ and  write  $b=\sum_{\lambda\in \Lambda}b_{\lambda}\lambda$ with coefficients $b_{\lambda}\in\mathbb{C}$.
We   extend Lemma \ref{lem:norm estimate for conjugation representation} to the relative setting.

\begin{lemma}
\label{lem:norm estimate for general representation}
Fix   $0 < \beta < 1$. Let $\varphi \in \relTr{\Lambda}{\Gamma}$ be a relative trace. Then
\begin{equation}
\label{eq:smile}
\|\alpha_\varphi(a)_{|\mathrm{span}_\mathbb{C}(v_\varphi)^\perp}\| \le \beta^{\frac{1}{2}}
\end{equation}
if and only if the inequality 
%then $\varphi \in \Ch{\Gamma}$.
\begin{equation}
\label{eq:smile2}
 \varphi\left(\sum_{\lambda}b_{\lambda} x^* x^\lambda -\beta x^* x    \right) \leq  (1-\beta) |\varphi(x)|^2
\end{equation}
holds true for every element $x \in \mathbb{C}\left[\Gamma\right]$.
\end{lemma}

\begin{proof}
Let $(\pi_\varphi,\rho_\varphi,\mathcal{H}_\varphi,v_\varphi)$ be the GNS data associated to the relative trace $\varphi$. Consider a vector $w \in \mathcal{H}_\varphi$ of the form  $w = \pi_\varphi(x) v_\varphi$ for some element $x \in \mathbb{C}\left[\Gamma\right]$. We obtain the estimate
\begin{align} 
\begin{split}
\label{eq:fun equation}
\left\langle \alpha_{\varphi}(\lambda)w,w\right\rangle &=\left\langle \alpha_{\varphi}(\lambda)\pi_{\varphi}(x)v_{\varphi},\pi_{\varphi}(x)v_{\varphi}\right\rangle = \left\langle U_\varphi (\lambda)\pi_{\varphi}(x) U_{\lambda}^{-1}v_{\varphi},\pi_{\varphi}(x)v_{\varphi}\right\rangle = \\
&= \left\langle \pi_{\varphi}(x^\lambda)v_{\varphi},\pi_{\varphi}(x)v_{\varphi}\right\rangle
=\varphi\left(x^* x^\lambda \right)
\end{split}
\end{align}
 for every element $\lambda \in \Lambda$. 
Note that the   estimate  in     Equation 
 (\ref{eq:fun equation}) is analogous to that of Equation (\ref{eq: inner product for conjugation rep}) with the difference of using the unitary representation $\alpha_\varphi$ of  the group $\Lambda$ instead of the conjugation representation $c_\varphi$ of the group $\Gamma$.
In particular
\begin{equation}
\label{eq:somewhat fun equation}
\|\alpha_{\varphi}(a)w\|^{2}=\left\langle \alpha_{\varphi}(b)w,w\right\rangle =\sum_{\lambda\in \Lambda} \left\langle b_{\lambda} \alpha_{\varphi}(\lambda)w,w\right\rangle =\varphi\left(\sum_{\lambda\in \Lambda}b_{\lambda}x^* x^\lambda \right).
\end{equation}
Once again, note that Equation (\ref{eq:somewhat fun equation})  is analogous to Equation \ref{eq: inner product for conjugation rep2} up to replacing the conjugation representation $c_\varphi$ with the representation $\alpha_\varphi$.
From this point onward the proof follows in  exactly the same way as that of Lemma \ref{lem:norm estimate for conjugation representation}.
\end{proof}

\begin{rmk}
The only property of the conjugation representation $c_\varphi$ and the representation $\alpha_\varphi$, which was used in Lemmas   \ref{lem:norm estimate for conjugation representation} and \ref{lem:norm estimate for general representation} respectively, to express the spectral gap via formulas of traces is the fact that the representation under consideration preserves the  group algebra $\C[\Gamma]$ sitting inside the Hilbert space $\mathcal{H}_\varphi$.
\end{rmk}

\subsection*{Finite-dimensional relative characters}

The correspondence taking a relative trace $\varphi \in \relTr{\Lambda}{\Gamma}$ to the unitary representation $U_\varphi : \Lambda \to \mathrm{U}(\mathcal{H}_\varphi) $ is continuous with respect to the pointwise convergence topology on traces and the Fell topology on unitary representations. The proof of this fact is mutatis mutandis the same as that of  Lemma \ref{lem:taking character to its GNS is Fell continuous}. Indeed  as  $\mathrm{Inn}(\Gamma) \le \theta(\Lambda)$ we get $\pi_\varphi(\Gamma)\subseteq U_\varphi(\Lambda)$ so that   the vector $v_\varphi \in \mathcal{H}_\varphi$ is  cyclic for the representation $U_\varphi$ as well. 
%The continuity therefore follows with the exact same argument as in the proof of Lemma \ref{lem:taking character to its GNS is Fell continuous}. 
%We obtain the following result:
\begin{prop}\label{prop:rel-chars-fd-isolated}
Let $\Lambda$ be a countable group admitting an action   $\theta:\Lambda \to \Aut{\Gamma}$ on the countable group $\Gamma$. Assume that $\Lambda$ has Kazhdan's property (T) and that $\mathrm{Inn}(\Gamma) \le \theta(\Lambda)$. Then  any finite-dimensional relative character $\varphi\in \relCh{\Lambda}{\Gamma}$   is an  isolated point of the set $\relCh{\Lambda}{\Gamma}$.
\end{prop}
\begin{proof}
Consider a sequence of relative characters $\varphi_n\in \relCh{\Lambda}{\Gamma}$ converging pointwise to a finite-dimensional relative character $\varphi\in \relCh{\Lambda}{\Gamma}$.
 As  $\mathrm{Inn}(\Gamma) \le \theta(\Lambda)$ and taking into account  Proposition \ref{prop:understanding the center}   we get
$N_\varphi' = U_\varphi(\Lambda)'$. On the other hand, since $\varphi$ is a relative character we have $N_\varphi'=\C \cdot \mathrm{Id}$. Therefore $U_\varphi(\Lambda)' = \mathbb{C} \cdot \mathrm{Id} $ which means that $U_\varphi(\Lambda)$ is an irreducible representation. The same argument applied with respect to each relative character $\varphi_n$ shows that each representation $U_{\varphi_n}$ is irreducible.

The  paragraph preceding this proposition shows that  the unitary representations corresponding to the relative characters $\varphi_n$ and $\varphi$ satisfy $U_{\varphi_n}\to U_\varphi$   in the Fell topology.
The assumption that the relative character $\varphi$ is finite-dimensional means that the Hilbert space $\mathcal{H}_\varphi$ is finite-dimensional. In particular, the representation $U_\varphi$ is finite-dimensional.
Recall that for Kazhdan groups, finite-dimensional representations are isolated points of the unitary dual in the Fell topology \cite[Theorem 1.2.5]{bekka2008kazhdan}. It follows that $U_{\varphi_n}=U_\varphi$ for   all $n\in \N$ sufficiently large. In particular $\pi_{\varphi_n}=\pi_\varphi$ so that $\varphi_n=\varphi$ for all $n \in \N$ sufficiently large.
\end{proof}

\subsection*{The simplex of relative traces}

Assume that the group $\Lambda$ has   property (T) and that $\mathrm{Inn}(\Gamma) \le \theta(\Lambda)$. The fact that the space of relative traces $\relTr{\Lambda}{\Gamma}$ is a Bauer simplex follows along the  same lines as   the proof of Theorem \ref{cor intro: Kazhdan group has Bauer simplex}. 
Let us provide a bit more detail.

\begin{proof}[Proof of Theorem \ref{thm intro: characteristic characters is closed}]
Consider a sequence of relative characters $\varphi_n \in \relCh{\Lambda}{\Gamma}$  converging pointwise to some relative trace  $\varphi \in \relTr{\Lambda}{\Gamma}$.
For each $n \in \N$  the subspace of  the Hilbert space $\mathcal{H}_{\varphi_n}$ consisting of ${\alpha_{\varphi_n}}$-invariant vectors is one-dimensional and is spanned by the cyclic vector $v_{\varphi_n}$, see  
 Proposition \ref{prop: relative character by dimension of invariant space}. Since the group $\Lambda$ has Kazhdan's property (T) and according to Lemma \ref{lem:norm estimate for general representation} there is some constant $0 < \beta < 1$  (depending on the chosen element $a \in \ell^1(\Lambda)$) such that Equation (\ref{eq:smile2}) holds true with respect to all the relative characters $\varphi_n$. The validity of this equation passes to the limiting trace $\varphi$ with the same value of $\beta$. Therefore the relative trace $\varphi$ is a relative character by the converse direction of Lemma \ref{lem:norm estimate for general representation} and of Proposition \ref{prop: relative character by dimension of invariant space}.
The second part of the statement concerning  finite-dimensional relative characters has already been established in Proposition \ref{prop:rel-chars-fd-isolated}.
% The main ingredient there was Proposition \ref{prop:Ch L is closed} combined with the purely abstract Lemma \ref{lem:Bauer-simplex-general}. To adapt it to the relative case, we need to adapt the argument for the representation $\alpha_\varphi$ instead of the conjugation representation $c_\varphi$. This is done by replacing 
% Lemma \ref{lem:norm estimate for conjugation representation}  with Lemma \ref{lem:norm estimate for general representation} in order to pass to the limit, and then using  instead of
% Corollary \ref{cor:c-invariant vectors are one-dimensional iff its a character} in order to conclude that the limiting relative character is again a character.
\end{proof}

The last part of our work deals with relative characters of the   abelian group $\mathbb{Z}^d$.

\begin{cor}\label{cor:torus-measures-limit}
Assume that $d\geq 3$. Any sequence of pairwise distinct relative characters $\varphi_n\in \relCh{\mathrm{SL}_d(\Z)}{\Z^d}$ converges pointwise to the Dirac trace $\delta_e$ on $\Z^d$.
\end{cor}
The Pontryagin dual of the free abelian group $\Z^d$ is the torus $\T^d$. Bochner's theorem sets up an $\mathrm{SL}_d(\mathbb{Z})$-equivariant affine homeomorphism $\mathrm{Tr}(\Z^d) \cong \mathrm{Prob}(\T^d)$ of compact convex sets,  where the space of probability measures $\mathrm{Prob}(\T^d)$ is equipped with the weak-$*$ topology. Therefore Corollary \ref{cor:torus-measures-limit} is equivalent to the following dynamical statement:
\emph{Let $\mu_n$ be any sequence of pairwise distinct ergodic $\mathrm{SL}_d(\Z)$-invariant Borel probability measures on the torus $\T^d$. Then the sequence $\mu_n$ converges in the weak-$*$ topology to the Haar measure on the torus $\T^d$.} This statement is certainly well known to experts and is valid for all $d \ge 2$ (however we could not locate a reference in the literature).

 \begin{proof}[Proof of Corollary \ref{cor:torus-measures-limit}]
Let $\varphi_n \in \relCh{\mathrm{SL}_d(\Z)}{\Z^d}$  be any sequence of pairwise distinct relative characters. Consider an accumulation point $\varphi \in \relTr{\mathrm{SL}_d(\Z)}{\Z^d}$ so that $\varphi$ is a relative trace.    Theorem \ref{thm intro: characteristic characters is closed} says that  $\varphi$ is in fact an infinite-dimensional relative character. In particular, the corresponding ergodic $\mathrm{SL}_d(\Z)$-invariant Borel probability measure $\widehat{\varphi} \in \mathrm{Prob}(\T^d)$ is non-atomic. However the only non-atomic $\mathrm{SL}_d(\Z)$-invariant Borel probability measures on the torus $\T^d$ is the Haar measure \cite[Proposition 9]{burger1991kazhdan}. In other words, the  character $\varphi$ must be the Dirac trace $ \delta_e$.
\end{proof}

% For each $n\in \N$, let $A_d(n)$ be the set of all $d$-tuples of the form $(\frac{x_1}{n},...,\frac{x_d}{n})$ where $x_1,...,x_d$ are of positive integers not greater than $n$, such that greatest common divisor of the set $\{x_1,...,x_d,n\}$ is $1$. It is not hard to see that when viewed as a subset of $\T^d$, 
% $A_n^d$ a finite orbit under the action of $\mathrm{SL}_d(\Z)$. We thus immidately get: 
% \begin{cor}
%     For $d\geq 3$, the sequence of sets $A_d(n)$ equidistributes in $[0,1]^d$
% \end{cor}

We conclude with a  vanishing result for the semidirect product $ \mathrm{Aut}(\Z^d) \ltimes \Z^d$.

\begin{thm}\label{thm:semidirect product}
    Let $\varphi_n\in \Ch{\mathrm{GL}_d(\Z)\ltimes \Z^d}$ be any sequence of characters such that the restrictions $\varphi_{n\,|\Z^d}$ are pairwise distinct. Then the characters $\varphi_n$ converge  pointwise to the Dirac character $\delta_e \in \Ch{\mathrm{GL}_d(\Z)\ltimes \Z^d}$.
\end{thm}

\begin{proof}
Assume that $d \ge 3$. Consider the semidirect product $\Gamma = \mathrm{GL}_d(\Z) \ltimes \Z^d$.   Let $\varphi_n \in \Ch{\Gamma}$ be any sequence of characters whose restrictions to the subgroup $\Z^d$ are pairwise distinct. Let $\varphi \in \Tr{\Gamma}$ be any accumulation point of the sequence $\varphi_n$. We know that  $\varphi(\gamma) = 0$ for any element $\gamma \in \Gamma \setminus \Rad{\Gamma}$ where $\Rad{\Gamma}$ is the amenable radical given by $\Rad{\Gamma} = \mathrm{Z}(\mathrm{GL}_d(\Z)) \ltimes \Z^d$, see Theorem \ref{thm:general arithmetic}. In addition, the fact that $\varphi(\gamma) = 0 $ for any element $\gamma \in \Z^d \setminus \{e\}$ follows from  Corollary \ref{cor:torus-measures-limit}. 

It remains to show that $\varphi(\gamma) = 0$ for all elements $\gamma \in \Rad{\Gamma} \setminus \Z^d$.
Note that any given element $\gamma \in \Rad{\Gamma} \setminus \Z^d$ is of the form $\gamma = (-\mathrm{Id}, u)$ for some $u \in \Z^d$. Take any sequence  of pairwise distinct elements $v_n \in \Z^d$. The commutators given by
\begin{equation}
\left[\gamma,v_n\right] = \left[(-\mathrm{Id}, u), (\mathrm{Id},v_n) \right] = (\mathrm{Id},2(v_n-u))
\end{equation}
are pairwise distinct as well. Since the trace $\varphi$ vanishes on $\Z^d \setminus \{e\}$ we conclude that $\varphi(\gamma) = 0$ using \cite[Lemma 4.13]{lavi2023characters}.
\end{proof}

\bibliographystyle{alpha}
\bibliography{musketeers}

\end{document}